\documentclass[oneside,notitlepage,11pt]{article}

\usepackage{amsthm}
\usepackage{thmtools}
\usepackage{graphicx}
\usepackage[margin=2cm,font=normalsize,labelfont=bf]{caption}
\usepackage{verbatim}
\usepackage[shortlabels]{enumitem}
\usepackage{fancyhdr}
\usepackage[]{geometry}
\usepackage{xstring}
\usepackage{needspace}
\usepackage{color}
\usepackage{amstext}
\usepackage{nameref}
\usepackage[hypertexnames=false]{hyperref}
\usepackage{mathdots}
\usepackage[normalem]{ulem}
\usepackage{chngcntr}
\usepackage[section]{placeins}
\usepackage{tikz-cd}    
  
\makeatletter
\newcounter{denseversion}
\newcounter{comments}
\newcounter{authorcounter}
\newcounter{adresscounter}

\def\title#1{\gdef\@title{#1}}
\def\@title{}
\def\subtitle#1{\gdef\@subtitle{#1}}
\def\@subtitle{}

\def\authortagsused{0}
\def\adresstag#1{\if!#1!\else$^{\;#1\;}$\fi}
\def\@authorsep#1{
  \ifnum\value{authorcounter}=#1 and \else\unskip, \fi
}

\renewcommand{\author}[2][]{
  \stepcounter{authorcounter}
  \if!#1!\else\gdef\authortagsused{1}\fi
  \ifnum\value{authorcounter}=1
    \def\@authorstringa{#2\adresstag{#1}}
    \def\@authorstringb{#2}
    \def\@authorstringc{#2\adresstag{#1}}
  \else
    \ifnum\value{authorcounter}=2
      \g@addto@macro\@authorstringa{\@authorsep{2}#2\adresstag{#1}}
      \g@addto@macro\@authorstringb{\@authorsep{2}#2}
      \g@addto@macro\@authorstringc{\\#2\adresstag{#1}}
    \else
      \g@addto@macro\@authorstringa{\@authorsep{3}#2\adresstag{#1}}
      \g@addto@macro\@authorstringb{\@authorsep{3}#2}
      \g@addto@macro\@authorstringc{\\#2\adresstag{#1}}
    \fi
  \fi}
\def\@author{\ifnum\value{denseversion}=0\@authorstringa\else\@authorstringb\fi}

\def\@adressstringa{}
\def\@adressstringb{}
\newcommand{\adress}[2][]{
  \stepcounter{adresscounter}
  \ifnum\value{adresscounter}=1
    \g@addto@macro\@adressstringa{\ifnum\authortagsused=0\def\br{\\}\else\def\br{, }\fi\adresstag{#1}#2}
    \g@addto@macro\@adressstringb{\def\br{\\}\adresstag{#1}\parbox[t]{14cm}{#2}}
  \else
    \g@addto@macro\@adressstringa{\\[\bigskipamount]\adresstag{#1}#2}
    \g@addto@macro\@adressstringb{\\[\medskipamount]\adresstag{#1}\parbox[t]{14cm}{#2}}
  \fi}

\def\preprint#1{\gdef\@preprint{#1}}
\def\@preprint{}
\def\keywords#1{\gdef\@keywords{#1}}
\def\@keywords{}
\def\msc#1{\gdef\@msc{#1}}
\def\@msc{}
\def\email#1{
   \gdef\@email{#1}
   \g@addto@macro\@authorstringc{ {\it (#1)}}}
\def\@email{}
\def\dedication#1{\gdef\@dedication{#1}}
\def\@dedication{}

\def\mybaselinestretch#1{
  \gdef\@mybaselinestretch{#1}
  \renewcommand{\baselinestretch}{\@mybaselinestretch}}

\def\myparskip#1{
  \gdef\@myparskip{#1}
  \setlength{\parskip}{\@myparskip}}

\mybaselinestretch{1.2}
\myparskip{1.5ex}

\binoppenalty=\maxdimen
\relpenalty=\maxdimen

\normalfont
\normalsize
\newlength{\@listleftmargin}
\def\setenumstandard{
  \setlist{leftmargin=\@listleftmargin,itemsep=0pt,topsep=0pt,partopsep=0pt,parsep=\@myparskip}
  \setlist[enumerate]{align=left,labelsep=*,leftmargin=\@listleftmargin,itemsep=0pt,topsep=0pt,partopsep=0pt,parsep=\@myparskip}
}

\def\denseversion{
  \setcounter{denseversion}{1}
  \newgeometry{left=3cm,right=3cm,top=3cm}
  \mybaselinestretch{1.1}
  \myparskip{0.8ex}
  \normalfont
  \def\possiblelinebreak{}
  \fancyfoot[C]{\itshape{--$\,\,$\thepage$\,\,$--}}}

\def\possiblelinebreak{\\}

\renewcommand{\emph}[1]{\def\reserved@a{it}\ifx\f@shape\reserved@a\ul{#1}\else\textit{#1}\fi}

\def\setcrefnames{}
\newcommand{\mytableofcontents}{
   \ifnum\value{denseversion}=0
     \tableofcontents
     \setcrefnames %% \tableofcontents macht die \crefnames kaputt
   \else
     \renewcommand{\baselinestretch}{1.1}
     \setlength{\parskip}{0ex}
     \normalfont
     \begingroup
     \def\addvspace##1{\vskip0.4em}
     \tableofcontents
     \setcrefnames %% \tableofcontents macht die \crefnames kaputt
     \endgroup
     \renewcommand{\baselinestretch}{\@mybaselinestretch}
     \setlength{\parskip}{\@myparskip}
     \normalfont
   \fi}

\newlength{\zeilenlaenge}
\def\putindent#1{
  \settowidth{\zeilenlaenge}{#1}
  \ifnum\zeilenlaenge>\textwidth
    #1
  \else
    \noindent #1
  \fi
}

\hypersetup{
    linktocpage = true,
    colorlinks,
    linkcolor=[RGB]{0,0,120},
    citecolor=[RGB]{0,0,120},
    urlcolor=[RGB]{0,0,120}
}

\def\pdfdaten{
  \hypersetup{
    pdftitle = {\@title},
    pdfauthor = {\@author},
    pdfkeywords = {\@keywords},    
    bookmarksopen = true,
    bookmarksopenlevel = 1
  }}  

\def\showkeywords{\begin{flushleft}\footnotesize\textbf{Keywords}: \@keywords\end{flushleft}}
\def\showmsc{\begin{flushleft}\footnotesize\textbf{MSC 2010}: \@msc\end{flushleft}}

\def\tocsection#1{\section*{#1}\addcontentsline{toc}{section}{#1}}

\def\mytitle{}
\def\zmptitle{
  \begin{tabular}{cc}
    \begin{minipage}[c]{0.4\textwidth}
      \begin{flushleft}
        \includegraphics[width=110pt]{../../tex/zmp}
      \end{flushleft}  
    \end{minipage}&
    \begin{minipage}[c]{0.55\textwidth}
      \begin{flushright}
      {\small\sf\@preprint}
      \end{flushright}
    \end{minipage}
  \end{tabular}
  \vskip 2cm}

\def\maketitle{
  \pdfdaten
  \noindent
  \mytitle
  \begin{center}
    \LARGE\@title\\
    \if!\@subtitle!\else\smallskip\LARGE\@subtitle\\\fi
    \bigskip
    \if!\@author!\else\bigskip\large\@author\\\fi
    \ifnum\value{denseversion}=0
      \if!\@adressstringa!\else\bigskip\normalsize\@adressstringa\\\fi
      \if!\@email!\else\ifnum\value{authorcounter}=1\bigskip\normalsize\textit{\@email}\\\else\fi\fi
    \else
    \fi
    \if!\@dedication!\else\bigskip\normalsize{\@dedication}\\\fi
  \end{center}
  \ifnum\value{denseversion}=0\vskip 1.5cm\else\vskip0.5cm\fi}

\def\kobib#1{
  \begin{raggedright}
  \ifnum\value{denseversion}=0\else\small\fi
  \Oldbibliography{#1/kobib}
  \bibliographystyle{#1/kobib}
  \end{raggedright}
  \ifnum\value{denseversion}=0\else
      \noindent
      \if!\@authorstringc!\else
        \ifnum\authortagsused=0\ifnum\value{authorcounter}>1\normalsize\@authorstringc\\[\medskipamount]\else\fi\else\normalsize\@authorstringc\\[\medskipamount]\fi
      \fi
      \if!\@adressstringb!\else\normalsize\@adressstringb\\{}\fi
      \ifnum\authortagsused=0
        \ifnum\value{authorcounter}=1
          \if!\@email!\else\linebreak\normalsize\textit{\@email}\\{}\fi
        \else
        \fi
      \else
      \fi
  \fi}

\let\Oldbibliography\bibliography
\def\bibliography#1{
  \begin{raggedright}
  \ifnum\value{denseversion}=0\else\small\fi
  \Oldbibliography{#1}
  \end{raggedright}
  \ifnum\value{denseversion}=0\else
      \medskip
      \noindent
      \if!\@authorstringc!\else
        \ifnum\authortagsused=0\ifnum\value{authorcounter}>1\normalsize\@authorstringc\\[\medskipamount]\else\fi\else\normalsize\@authorstringc\\[\medskipamount]\fi
      \fi
      \if!\@adressstringb!\else\normalsize\@adressstringb\\{}\fi
      \ifnum\authortagsused=0
        \ifnum\value{authorcounter}=1
          \if!\@email!\else\linebreak\normalsize\textit{\@email}\\{}\fi
        \else
        \fi
      \else
      \fi
  \fi
}

\newenvironment{commentfigure}{\begin{comment}}{\end{comment}}
\newenvironment{sidewayscommentfigure}{\begin{minipage}}{\end{minipage}}
\newenvironment{displaycomment}{\begin{list}{}{\rightmargin=1cm\leftmargin=1cm}\item\sf\begin{small}\color{gray}}{\end{small}\end{list}}

\def\tocmark#1{}

\def\draftstamp#1{
  \def\tocmark##1{
    \ifnum\c@secnumdepth=0\section{##1}\fi
    \ifnum\c@secnumdepth=1\subsection{##1}\fi
    \ifnum\c@secnumdepth=2\subsubsection{##1}\fi
    \ifnum\c@secnumdepth=3\subsubsection{##1}\fi
  }
  \ifnum\value{comments}=0
    \gdef\@draft{DRAFT - Edited on \today\ by #1 - Comments are not displayed}
  \else
    \gdef\@draft{DRAFT - Edited on \today\ by #1 - Comments are displayed}
  \fi
  \fancyhead[C]{\footnotesize\tt\textcolor{red}{\@draft}}}

\def\skript{
  \renewenvironment{displaycomment}{}{}
  \ifnum\value{comments}=0
    \renewenvironment{example*}{\comment}{\endcomment}
    \renewenvironment{remark*}{\comment}{\endcomment}
  \else\fi
  \parindent=0mm        
}

\mybaselinestretch{}
\setenumstandard

\makeatother

\input{kobib.tex}
\usepackage{amssymb}
\usepackage{amsmath}
\usepackage{amstext}
\usepackage{mathrsfs}
\usepackage{euscript}
\usepackage[all,cmtip]{xy}
\usepackage{xstring}
\usepackage{slashed}
\usepackage{tensor}
\usepackage{mathtools}

\def\ul{\underline}
\entrymodifiers={+!!<0pt,\fontdimen22\textfont2>}

\def\N {\mathbb{N}}
\def\Z {\mathbb{Z}}

\def\R {\mathbb{R}}
\def\C {\mathbb{C}}

\def\h {\mathrm{H}}

\def\subset{\subseteq}
\def\supset{\supseteq}

\def\inv{\mathrm{inv}}

\makeatletter
\renewenvironment{proof}[1][\nameProof]
  {\par\pushQED{\qed}%
   \normalfont \topsep6\p@\@plus6\p@\relax
   \trivlist
   \item[\hskip\labelsep
         \itshape
         #1\@addpunct{.}]
  \leavevmode}
  {\popQED\endtrivlist\@endpefalse}
\makeatother

\def\notebox#1#2{\begin{minipage}[b]{#1}\sloppy\renewcommand{\baselinestretch}{0.8}\footnotesize \begin{center}#2\end{center}\end{minipage}}

\def\mquad{\hspace{-2em}}

\newcommand{\arr}[1][r]{\ar@<0.7ex>[#1]\ar@<-0.7ex>[#1]}
\newcommand{\arrr}[1][r]{\ar@<1.4ex>[#1]\ar[#1]\ar@<-1.4ex>[#1]}

\newlength{\myeqt} % Darin wird die Länge des übergebenen Textes abgespeichert
\newlength{\myeqs} % Darin wird die Länge des übergebenen Symbolds abgespeichert
\newlength{\myeqm} % Wieviel "klein" über die Breite des Symbolds hinausgehen darf
\newlength{\myeqn} % Die Standardbreite für große Boxen
\newcommand\symtext[3][\myeqn]{
  \settowidth{\myeqt}{#2}
  \settowidth{\myeqs}{$#3$}
  \addtolength{\myeqs}{\the\myeqm}
  \ifdim\myeqt>\myeqs
    \stackrel{\hspace{-#1}\notebox{#1}{\medskip #2 \\ $\downarrow$\smallskip}\hspace{-#1}}{#3}
  \else
    \stackrel{\text{#2}}{#3}
  \fi}

\def\brackets#1{\IfStrEq{#1}{-}{}{(#1)}}
\def\subindex#1{\IfStrEq{#1}{-}{}{_{#1}}}

\newcommand{\alxydim}[2]{\begin{aligned}\xymatrix#1{#2}\end{aligned}}

\makeatletter
\def\bigset#1#2{\left\lbrace\;\begin{minipage}[c]{#1}\begin{center}#2\end{center}\end{minipage}\;\right\rbrace}

\makeatother

\newlength{\myl}

\def\ddt#1#2#3{\left.\frac{\mathrm{d}^{\IfStrEq{#1}{1}{}{#1}}}{\mathrm{d}#2}\IfStrEq{#2}{#3}{\right.}{\right|_{#3}}}

\def\diff{\mathscr{D}\mathrm{iff}}

\def\frech{\mathscr{F}\mathrm{rech}}
\def\top{\mathscr{T}\mathrm{op}}

\def\B{\mathcal{B}\hspace{-0.03em}}

\newlength{\widthtmp}
\def\length#1{\settowidth{\widthtmp}{#1}\the\widthtmp}

\definecolor{olivegreen}{rgb}{.33,.55,.18}

\newcommand{\ie}{i.e., }

\renewcommand{\O}{\operatorname{O}}

\newcommand{\SO}{\operatorname{SO}}
\newcommand{\Spin}{\operatorname{Spin}}
\newcommand{\String}{\operatorname{String}}
\newcommand{\XString}{\mathscr{S}\mathrm{tring}}
\newcommand{\GString}{\mathcal{S}tring}

\newcommand{\U}{\operatorname{U}}
\newcommand{\PU}{\operatorname{PU}}

\newcommand{\Cl}{\operatorname{Cl}}

\newcommand{\opp}{\operatorname{op}}
\newcommand{\pr}{\operatorname{pr}}

\newcommand{\id}{\operatorname{id}}

\newcommand{\res}{\mathrm{res}}
\newcommand{\dis}{dis}

\newcommand{\Aut}{\mathrm{Aut}}
\newcommand{\Out}{\mathrm{Out}}
\newcommand{\BB}{\mathcal{B}}
\newcommand{\AUT}{\mathscr{A}\mathrm{ut}}

\newcommand{\Imp}{\mathrm{Imp}}

\newcommand{\Alg}{\Incl\mathrm{lg}}

\def\UAUT{\mathscr{U}}

\def\vNAlg{\mathrm{v}\mathscr{N}\Alg}

\newcommand{\Incl}{\mathscr{A}}

\usepackage[latin1]{inputenc}
\usepackage[english]{babel}
\usepackage[english]{cleveref}

\def\refname{References}

\def\quot#1{``#1''}

\def\quand{\quad\text{ and }\quad}

\def\nameTheorem{Theorem}
\def\nameTheorems{Theorems}
\def\nameDefinition{Definition}
\def\nameDefinitions{Definitions}
\def\nameCorollary{Corollary}
\def\nameCorollaries{Corollaries}
\def\nameProposition{Proposition}
\def\namePropositions{Propositions}
\def\nameConstruction{Construction}
\def\nameConstructions{Constructions}
\def\nameLemma{Lemma}
\def\nameLemmas{Lemmas}
\def\nameRemark{Remark}
\def\nameRemarks{Remarks}
\def\nameProblem{Problem}
\def\nameProblems{Problems}
\def\nameExample{Example}
\def\nameExamples{Examples}
\def\nameExercise{Exercise}
\def\nameExercises{Exercises}

\def\nameProof{Proof}
\def\nameEquation{\unskip}
\def\nameEquations{\unskip}
\def\nameSection{Section}
\def\nameSections{Sections}
\def\nameAppendix{Appendix}
\def\nameAppendixs{Appendices}
\def\nameFigure{Figure}
\def\nameFigures{Figures}

\input{theorems}

\def\mathscr#1{\EuScript{#1}}

\usepackage{tikz}
\usepackage{tikz-cd}
\title{A representation of the string 2-group}

\author[a]{Peter Kristel}
\email{pkristel@gmail.com}

\adress[a]{
        Hausdorff Center for Mathematics\\
        Endenicher Allee 62 \\ 
        53115 Bonn}

\author[b]{Matthias Ludewig}
\email{matthias.ludewig@mathematik.uni-regensburg.de}

\adress[b]{
        Universit\"at Regensburg\\
        Fakult\"at f\"ur Mathematik\\
93053 Regensburg}

\author[c]{Konrad Waldorf}
\email{konrad.waldorf@uni-greifswald.de}

\adress[c]{
        Universit\"at Greifswald\\
        Institut f\"ur Mathematik und Informatik\\
   17487 Greifswald}

\msc{}
\dedication{}
\denseversion
\allowdisplaybreaks
\usepackage{amstext}

\newcommand{\Rep}{R}

\begin{document}

        \maketitle

\begin{abstract}
We construct a representation of the string 2-group on a 2-vector space, aiming to establish it as the categorification of the spinor representation. Our model for 2-vector spaces is based on the Morita bicategory of von Neumann algebras, and we specifically represent the string 2-group on the hyperfinite type III$_1$ factor.
\end{abstract}

\mytableofcontents
\setsecnumdepth{1}

\section{Introduction}

The string group $\String(d)$, $d\geq 5$, is defined, up to homotopy equivalence, as the stage in the Whitehead tower of the orthogonal group following $\Spin(d)$,
\begin{equation*}
\cdots \to \String(d) \to \Spin(d) \to \SO(d) \to \O(d).
\end{equation*}
While $\String(d)$ can be given the structure of a topological group, it cannot be equipped with the structure of a finite-dimensional Lie group and in recent years, the insight emerged that it is geometrically most fruitful to realize $\String(d)$ as a categorified  group, or \emph{2-group} \cite{baez5}, a point of view that will be further advocated in this paper.

%\paragraph{String group models.}

Several models of the string 2-group in different contexts have been constructed, e.g., as a strict Fr\'echet Lie 2-group \cite{baez9}, as a finite-dimensional smooth \quot{stacky} 2-group \cite{pries2}, or as a strict diffeological 2-group \cite{Waldorf}. 
In this paper, using a combination of the models of \cite{baez9} and \cite{Waldorf}, we model the string 2-group as a strict Fr\'echet Lie 2-group $\XString(d)$, which, owing to its strictness, is very convenient to work with.

From the beginning, it was a major question if any model of the string 2-group would support a representation, meant to be a categorification of the spinor representation of the spin groups. 
This question is embedded in the  challenge to develop a good representation theory for general 2-groups, starting with a sensitive choice for a categorification of a vector space, a \emph{2-vector space}. 
Indeed, many proposals for 2-vector spaces have been studied from a representation-theoretic perspective: Kapranov-Voevodsky 2-vector spaces \cite{kapranov1}, see e.g., \cite{Barrett2004,Elgueta2007,Ganter2008,Ganter2014},  Baez-Crans 2-vector spaces \cite{baez7,Heredia}, or Crane-Yetter's measurable categories \cite{Crane2005,Baez2012}. 
Unfortunately, none of these frameworks for representations of 2-groups turned out to be flexible enough to contain a representation of the string 2-group, though various speculations of attempts have been reported, e.g. in \cite{Baez2012,Nikolausb,Murray2017a}.

%\paragraph{2-group representations.}

Following ideas of Schreiber \cite{Schreiber2006,Schreiber2007,Schreiber2009,schreiber2}, the work of Stolz and Teichner \cite{ST04,stolz3}, and our previous work in the finite-dimensional context \cite{Kristel2020}, in this article, we propose to use the bicategory of von Neumann algebras, bimodules, and intertwiners as a model for 2-Hilbert spaces. 
In a sense, our setting can be regarded as a \quot{non-abelian} generalization of Crane-Yetter's 2-group representations on measurable categories, as the latter ones can be seen as representation categories of \emph{abelian} von\  Neumann algebras \cite{Baez2012}.

A particular feature of our model for 2-Hilbert spaces is that the automorphism 2-group of a von Neumann algebra $A$ can be realized as a strict topological 2-group $\UAUT(A)$, the unitary automorphism 2-group of $A$, see \cref{DefCrossedModuleAUT}. 
A unitary representation of a  Lie 2-group $\mathscr{G}$ on a von Neumann algebra $A$ is then simply a  continuous 2-group homomorphism
\begin{equation*}
\mathscr{R}: \mathscr{G} \to \UAUT(A)\text{,}
\end{equation*} 
see \cref{representaion-of-discrete-2-group,continuous-unitary-representation}.

We remark that the unitary automorphism 2-group $\UAUT(A)$ of a von Neumann algebra $A$ seems \emph{not} to have the structure of a \emph{Lie} 2-group, and correspondingly we do expect unitary representations to be smooth.
This should not be a surprise, as even representations of ordinary finite-dimensional Lie groups on Hilbert spaces are typically not smooth but only strongly continuous.
 That $\UAUT(A)$ does not have a smooth structure but only the structure of a topological (in fact, polish) strict 2-group can therefore be seen as the higher-categorical analog of the fact that the unitary group $\U(H)$ of a Hilbert space $H$ in its strong topology is not a Lie group, but only a topological (in fact, polish) group.

%\paragraph{The stringor representation.}

In this paper we solve the long-standing open problem to construct a unitary representation
\begin{equation}
\label{stringor-representation}
\mathscr{R} : \XString(d) \to \UAUT(A)
\end{equation}
of the string 2-group,
where $A$ is a particular von Neumann algebra, namely the hyperfinite factor of type III$_1$.
To our best knowledge, this is the first time that a representation for any model of the string group has been established.

Our construction is based on many original ideas of Stolz and Teichner \cite{ST04,stolz3}. 
It became possible since we found suitable explicit models for both the string 2-group $\XString(d)$ and the III$_1$-factor $A$,  on the common basis of the theory of infinite-dimensional Clifford algebras and Fock spaces, developed by Araki \cite{Ara85}, Pressley-Segal \cite{PressleySegal}, Plymen-Robinson \cite{PR95}, and many others. 
This connection between higher differential geometry and von Neumann algebras was put together by the first-named author in his thesis \cite{kristel2019b}, which has been foundational to our ongoing research on  string geometry.

We call our representation $\mathscr{R}$ in \cref{stringor-representation} the \emph{stringor representation}.
This name is motivated by the idea that $\mathscr{R}$ is the higher analog of the spinor representation of the spin group. In the following we describe a result that supports this idea.

%\paragraph{Significance in string geometry.}

In spin geometry, the spinor bundle on a smooth $d$-dimensional manifold $M$ is a vector bundle obtained by applying the associated vector bundle construction to a spin structure on $M$ (a lift of the structure group of $M$ to $\Spin(d)$) and  the spinor representation of $\Spin(d)$.
When seeking
for analogous structures meeting the demands of string theory,  most successful has been the principle
established by Killingback and Witten to look at spin structures on the configuration space of strings
in $M$, the free loop space $LM = C^{\infty}(S^1,M)$ \cite{killingback1,witten2}. Such spin structures are different from
the ones mentioned above, because now all groups are infinite-dimensional. Nonetheless, Stolz and
Teichner have outlined a construction of an infinite-dimensional spinor bundle on $LM$ \cite{stolz3}. 
Moreover,
they established the principle of \emph{fusion} in loop space, expressing the idea that the relevant geometric
structures on loop space correspond to geometric structures on $M$ itself. In obvious
analogy, they coined the terminology \emph{stringor bundle} for the corresponding -- at that time, unknown -- structure on $M$.

A rigorous construction of the spinor bundle on  $LM$ and its fusion product has been given by the first- and third-named authors in a diffeological setting of \emph{rigged von Neumann algebra bundles} \cite{kristel2019b, Kristel2019,kristel2020smooth, Kristel2019c}, under the assumption that $M$ is equipped with a string structure $\mathscr{P}$, i.e., a principal $\XString(d)$-2-bundle $\mathscr{P}$ lifting the structure group of $M$ to $\XString(d)$.  These constructions have recently been simplified by the second-named author to a topological setting with locally trivial von Neumann algebra bundles \cite{ludewig2023spinor}.  
In our common article  \cite{StringRep}  we described how the fusion products allow to \emph{regress} the spinor bundle on loop space to a von Neumann 2-vector bundle $\mathscr{S}_{\mathscr{P}}$ on $M$, thus realizing the anticipated stringor bundle $\mathscr{S}_{\mathscr{P}}$ of Stolz and Teichner as a well-defined object in that setting. 

In  \cite{StringRep}, we describe a categorification of the associated vector bundle construction:  it associates to a principal $\mathscr{G}$-2-bundle $\mathscr{P}$ over a smooth manifold $M$ and a unitary representation $\mathscr{R}: \mathscr{G} \to \UAUT(A)$ on a von Neumann algebra $A$, a von Neumann 2-vector bundle $\mathscr{P} \times_{\mathscr{G}} A$ over $M$. 
There, we prove the following result:

\begin{theorem}[{{\cite{StringRep}}}]
Let $M$ be a string manifold with string structure $\mathscr{P}$.
Then, the stringor bundle $\mathscr{S}_{\mathscr{P}}$ is the associated 2-vector bundle for the string structure $\mathscr{P}$ and the stringor representation $\mathscr{R}:\XString(d) \to \UAUT(A)$, i.e., there is a canonical isomorphism of von Neumann 2-vector bundles 
\begin{equation*}
\mathscr{S}_{\mathscr{P}} \cong \mathscr{P} \times_{\XString(d)} A\text{.}
\end{equation*}
\end{theorem}

The above theorem is our current justification for the terminology \quot{stringor representation}: it exhibits the stringor bundle as a higher-geometric analogy of the spinor bundle. 
In ongoing research we will further investigate and try to advance this analogy. 

The present paper is an excerpt and improvement of our paper \cite{StringRep}, in which our construction of the stringor representation was carried out first in a diffeological setting, together with other constructions and results partially explained above.

\paragraph{Acknowledgements. } We would like to thank Severin Bunk, Andr\'e Henriques, and Peter Teichner for helpful discussions. 
PK gratefully acknowledges support from the Pacific Institute for the Mathematical Sciences in the form of a postdoctoral fellowship. 
ML gratefully acknowledges support from SFB 1085 ``Higher invariants''.

\section{Crossed modules and their representations}
\label{sec:2groups}

In the present paper, we only consider \emph{strict} 2-groups, and treat  those via their crossed modules. In \cref{2-group-perspective} we describe the corresponding treatment as groupoids internal to the category of groups.  

\begin{definition}
\label{crossed-module}
A \emph{crossed module} is a quadruple $\mathscr{G}=(G,H,t,\alpha)$ consisting of  groups $G$ and $H$, a  group homomorphism $t: H \to G$, and a  map $\alpha: G \times H \to H$, such that $\alpha$ is an action of $G$ on $H$ by group homomorphisms, and
\begin{equation} \label{CrossedModuleActions}
t(\alpha(g,h))=gt(h)g^{-1}
\quand
\alpha(t(h),k)=hkh^{-1}
\end{equation}
hold for all $g\in G$ and $h, k\in H$. 
A \emph{strict intertwiner} $\mathscr{R}:\mathscr{G} \to \mathscr{G}'$ between crossed modules $\mathscr{G}=(G,H,t,\alpha)$ and $\mathscr{G}'=(G',H',t',\alpha')$ is a pair $\mathscr{R}=(R_0,R_1)$ consisting of  group homomorphisms $R_0:G \to G'$ and $R_1:H \to H'$  such that
\begin{equation}
\label{CompatibilityStructureMaps}
R_0 (t(h)) = t' (R_1(h))
\quand
R_1(\alpha(g,h))= \alpha'(R_0(g),R_1(h))
\end{equation} 
hold for all $h\in H$ and $g\in G$.
\end{definition}

\begin{example}
\label{ex:2groupsAbelian}
Given any abelian group $A$ we consider the  crossed module  $A \to \{e\}$, with the (necessarily trivial) action, which we denote by $\mathscr{B}A$.
Observe that $A$ must be abelian because of the second equality in \cref{CrossedModuleActions}.
\end{example}

\begin{example}
\label{ex:2-groupsOrdinaryGroups}
Any group $G$ can be viewed as a crossed module $\{e\} \to G$, which we denote by $G_{\dis}$.
\end{example}

\begin{example}
\label{ExampleAUTA}
Let $A$ be an algebra (which we take to be over $\R$ or $\C$, and require it to be unital and associative).
Its \emph{automorphism 2-group} $\AUT(A)$ is the crossed module
\begin{equation*}
    A^\times \stackrel{t}{\to} \Aut(A),
\end{equation*}
where  $t$ assigns to a unit $a \in A^\times$ the corresponding inner automorphism of $A$, while the group $\Aut(A)$ of unital algebra automorphisms of $A$ acts on  $A^\times$ by evaluation.
 \end{example}

Foundational to this article is the following terminology, based on our choice to consider algebras as 2-vector spaces.  

\begin{definition}
\label{representaion-of-discrete-2-group}
A representation of a crossed module $\mathscr{G}$ on an algebra $A$ is a strict intertwiner 
\begin{equation*}
\mathscr{R}: \mathscr{G} \to \AUT(A)\text{.}
\end{equation*}
\end{definition}

\begin{remark}
\label{remark-about-automorphisms}
We proved in \cite[Prop. 2.3.1]{Kristel2022} that $\AUT(A)$ is the automorphism 2-group of the object $A$ in the Morita bicategory of algebras, if $A$ is Picard-surjective. 
In \cite[Prop. A.2]{Kristel2022} we proved that every algebra is isomorphic (in the Morita bicategory) to a Picard-surjective one. Thus,  \cref{representaion-of-discrete-2-group} is the natural definition of a representation on an object of a bicategory.  
\end{remark}

Most of the time we consider crossed modules  in the context of topological groups (\quot{topological crossed modules}) or (possibly Fr\'echet) Lie groups (\quot{Lie crossed modules}). 
In both cases it is straightforward to adapt \cref{crossed-module} to that setting: 
one requires the groups $G$ and $H$ to be topological groups (respectively Lie groups) and all structure maps to be continuous (respectively smooth).
Every Lie crossed module has an underlying topological crossed module, obtained by forgetting the smooth structure.

For instance, the crossed modules $\mathscr{B}A$ and $G_{\dis}$ are topological crossed modules if $A$ and $G$ are topological groups, and if $A$ is a finite-dimensional algebra, then $A^\times$ and $\Aut(A)$ have natural Lie group structures, which turn $\AUT(A)$ into a Lie crossed module.
Our \cref{representaion-of-discrete-2-group} thus generalizes immediately to continuous (smooth) representations of topological (Lie) crossed modules.  

Another  extension of \cref{representaion-of-discrete-2-group} is to $\ast$-algebras. If $A$ is a  $*$-algebra, its \emph{unitary automorphism 2-group} $\UAUT(A)$ is described by the crossed module
\begin{equation*}
   \U(A) \stackrel{t}{\to} \Aut^{*}(A),
\end{equation*}
where $\U(A)$ is the group of unitary elements in $A$ (i.e., those $u \in A$ with $uu^* = u^*u = 1$) and $\Aut^{*}(A)$  denotes the group of $*$-automorphisms.
The structure maps are just the restrictions of those in \cref{ExampleAUTA}.
A unitary representation is then a strict intertwiner $\mathscr{R}: \mathscr{G} \to \mathscr{U}(A)$. If $A$ is finite-dimensional, then $\mathscr{U}(A)$ is again a Lie crossed module, and we can look at smooth or continuous unitary representations.

\begin{comment}
We may also impose further restrictions on the topology of our crossed modules; for instance, consider crossed modules of Polish groups, of delta-generated groups, or of compactly-generated weak Hausdorff groups. 
\end{comment}

Just in the way that algebras are considered to be 2-vector spaces, it makes sense to consider %\cstar-algebras as \quot{inner product 2-vector spaces}, and 
von Neumann algebras as \quot{2-Hilbert spaces}. 
We will be interested in unitary representations on von Neumann algebras. 
In the following, we will adapt the above definition of the unitary automorphism 2-group $\UAUT(A)$ to von Neumann algebras, in particular taking care for the topological aspects.

We suppose that $A$ is a von Neumann algebra, which we assume to be representable on a separable Hilbert space.
We denote by $\U(A)$ its unitary group and, as before,  by $\Aut^{*}(A)$ the group of $*$-automorphisms of $A$. 
\begin{comment}
These automorphisms are automatically ultraweakly continuous Theorem 7.1.12 of Kadison-Ringrose
\end{comment}
In particular, inner automorphisms obtained by conjugation with a unitary $u \in \U(A)$ lie in $\Aut^*(A)$. 
The groups $\Aut^{*}(A)$ and $\U(A)$ have canonical topologies turning them into topological groups.
The group $\Aut^{*}(A)$ carries Haagerup's $u$-topology \cite[\S3]{HaagerupStandardForm}, which is the restriction of the topology on the set $\BB(A)$ of bounded operators on $A$ that is induced by the collection of seminorms
\begin{equation} \label{Definitionutopology}
        \| \theta \|_\xi = \|\xi \circ \theta \|_{A_*} = \sup_{\|a\|=1} |\xi (\theta(a))|,\qquad \xi \in A_*,
\end{equation}
on $A$ where $A_*$ denotes the predual of $A$.
The group $\U(A)$ carries the ultraweak topology.
The map $\U(A) \to \Aut^{*}(A)$ assigning to $u\in \U(A)$ the inner automorphism $a \mapsto uau^{*}$, as well as the evaluation action of $\Aut^{*}(A)$ on $\U(A)$ are continuous with respect to this topology. 
Now we are in position to introduce the following definitions.

\begin{definition}
\label{DefCrossedModuleAUT}
The \emph{unitary automorphism 2-group} $\UAUT(A)$ of a von Neumann algebra $A$ is the topological crossed module
\begin{equation*}
  \U(A) \stackrel{t}{\to} \Aut^{*}(A),
\end{equation*}
where $t$ assigns to $u \in \U(A)$ the inner automorphism given by conjugation with $u$ and $\Aut^{*}(A)$ acts on $\U(A)$ by evaluation.
\end{definition}

\begin{definition}
\label{continuous-unitary-representation}
A \emph{unitary representation} of a Lie or topological crossed module $\mathscr{G}$ on a von Neumann algebra $A$ is a continuous strict intertwiner 
\begin{equation*}
\mathscr{R}: \mathscr{G} \to \UAUT(A)\text{.}
\end{equation*}
\end{definition}

\begin{remark} \label{RemarkAutomorphism2GroupBimodules}
Analogously to \cref{remark-about-automorphisms},  a von Neumann algebra $A$ is an object in the bicategory $\vNAlg$ of von Neumann algebras, bimodules, and intertwiners, with the Connes fusion product as the composition of 1-morphisms, see \cite{Brouwer2003}.
  Hence, as an object in a bicategory, $A$ has an automorphism 2-group $\AUT_{\vNAlg}(A)$, here viewed as a monoidal category, whose objects are all $A$-$A$-bimodules $M$ that are invertible with respect to the Connes fusion product.
 When also $\UAUT(A)$ is viewed as a monoidal category, there is a canonical monoidal functor   $\UAUT(A) \to \AUT_{\vNAlg}(A)$ sending an automorphism $\theta\in \AUT^{*}(A)$ to $L^2(A)_\theta$, the  $\theta$-twisted standard form of $A$.
  For a general von Neumann algebra, this functor  is   fully faithful.
  \\
  If $A$ is a factor of type III (as later in \cref{SectionOperatorModels}), the functor $\UAUT(A) \to \AUT_{\vNAlg}(A)$ is moreover essentially surjective, and hence an equivalence of monoidal categories. 
  Indeed, let $M$ be an invertible $A$-$A$-bimodule. 
  By the Murray-von Neumann classification of type III factors, there is, up to isomorphism, a unique non-zero countably generated left $A$-module, hence as a left-module, $M$ must be isomorphic to the underlying left module of a standard form $L^2(A)$ for $A$. 
  That $M$ is invertible implies that the homomorphism $A^{\opp} \to A^\prime$ provided by the right action of $A$ on $M$ must by an isomorphism.
  Comparing the right action induced by the modular conjugation of $L^2(A)$ with the right action therefore provides an automorphism $\theta \in \Aut^{*}(A)$ such that $M \cong L^2(A)_\theta$. 
  Summarizing, for a type III factor $A$, the automorphism 2-group $\UAUT(A)$ is a strict and topological version of the automorphism 2-group of $A$ as an object in the bicategory $\vNAlg$. 
\end{remark}

  If $\mathscr{G}$ is a topological crossed module,  we define the topological groups   $\pi_1\mathscr{G} := \mathrm{ker}(t)$, equipped with the subspace topology, and  $\pi_0\mathscr{G} := G/t(H)$, equipped with the quotient topology. Every continuous strict intertwiner  $\mathscr{R}: \mathscr{G} \to \mathscr{G}^\prime$  induces continuous group homomorphisms $\pi_0\mathscr{R}: \pi_0\mathscr{G} \to \pi_0\mathscr{G}^\prime$  and $\pi_1 \mathscr{R}: \pi_1\mathscr{G} \to \pi_1\mathscr{G}^\prime$. 
 In particular, any unitary representation $\mathscr{R}: \mathscr{G}\to \UAUT(A)$ induces continuous group homomorphisms $\pi_1\mathscr{R}: \pi_1\mathscr{G} \to Z(\U(A)) = \U(Z(A))$ (in particular, a unitary representation of $\pi_1\mathscr{G}$ on $Z(A)$) and $\pi_0\mathscr{R}:\pi_0\mathscr{G} \to \mathrm{Out}^{*}(A)$ (an \quot{outer representation} of $\pi_0\mathscr{G}$ on $A$). 
 
We may  consider the topological crossed modules $\mathscr{B}\pi_1\mathscr{G}$ and $(\pi_0\mathscr{G})_{dis}$ , and the  sequence
\begin{equation}
\label{eq:extension}
   \mathscr{B} \pi_1\mathscr{G} \to \mathscr{G} \to (\pi_0\mathscr{G})_{dis}
\end{equation} 
of strict intertwiners given by inclusion and projections, respectively
%
\begin{comment}
Set-theoretically, this sequence is an extension of 2-groups in the sense of Vitale \cite{Vitale2002}.
It is not \quot{strictly exact} in the sense of \cite{baez9}, but we are sure that \cref{eq:extension} must be exact in a non-strict sense, for which we have not found a reference.
\end{comment}
%
The action $\alpha$  induces an action of $\pi_0\mathscr{G}$ on $\pi_1\mathscr{G}$, and following \cite{pries2}  we call $\mathscr{G}$ and the corresponding extension \cref{eq:extension} \emph{central} if that action vanishes. In case of the unitary automorphism group of a von Neumann algebra $A$ (\cref{DefCrossedModuleAUT}), we have $\pi_1\UAUT(A)=Z(\U(A))=\U(Z(A))$, the group of central unitaries, and $\pi_0\UAUT(A)=\Out^{*}(A) := \Aut^{*}(A)/\U(A)$. Moreover, if $A$ is a factor, then $\pi_1\UAUT(A)=\U(1)$ and $\UAUT(A)$ is central.

\begin{remark}
Baez-Lauda have classified non-continuous, i.e., purely set-theoretic, extensions of the form \cref{eq:extension} by the group cohomology $\h^3_{\mathrm{grp}}(\pi_0\mathscr{G},\pi_1\mathscr{G})$ of $\pi_0\mathscr{G}$ with values in the module $\pi_1\mathscr{G}$ \cite{baez5}.  The class $k_{\mathscr{G}}$ corresponding to a  crossed module $\mathscr{G}$ is called the \emph{k-invariant} of $\mathscr{G}$.
For a continuous central crossed module $\mathscr{G}$ such that $\pi_0\mathscr{G}$ is Hausdorff, paracompact and locally compact, 
extensions of the form \cref{eq:extension} are classified by continuous Segal-Mitchison cohomology $\h^3_{\mathrm{SM}}(\pi_0\mathscr{G},\pi_1\mathscr{G})$ \cite{Blanco2020}. 
Depending on $A$, the topology of $\pi_0\UAUT(A)=\Out^{*}(A)$ can be very bad, for instance, non-Hausdorff. 
This happens, in particular, for the hyperfinite type III factor we use later in \cref{SectionOperatorModels}, where the quotient topology on $\Out^{*}(A)$ is indiscrete. 
Hence, we  cannot to use this  k-invariant of the unitary automorphism group $\UAUT(A)$.
\end{remark}

\section{The string 2-group}
\label{sec:string2group}

We model the string 2-group $\XString(d)$ as a crossed module of nuclear Fr\'echet Lie groups. 
It is a variation of (and equivalent to) a crossed module that appeared first in \cite{baez9}; see \cite{LudewigWaldorf2Group} for a detailed discussion.

To start with, for a Lie group $G$ with identity element $e$, we denote by $LG = C^\infty(S^1, G)$ the smooth loop group of $G$, and for subsets $I \subset S^1$, we write $L_I G$ for subgroup of $\gamma \in LG$ such that $\gamma(t) = e$ for $t \notin I$, respectively $t \in [0, \pi]$.
Moreover, we denote by $P_e G$ the space of smooth paths $\gamma : [0, \pi] \to G$ that are \emph{flat} at the end points, i.e., all derivatives of $\gamma$ at $t=0, \pi$ vanish.
We note that the elements of $L_{[0, \pi]} G$ are always flat at $t=0, \pi$; hence, restriction gives an injective group homomorphism 
\begin{equation}
\label{InclusionLoopsPaths}
r : L_{[0, \pi]} G \to  P_e G, \qquad \gamma \longmapsto \gamma|_{[0, \pi]}.
\end{equation}
All  groups discussed above are infinite-dimensional Lie groups modeled on nuclear Fr\'echet spaces.

Let $G$ be a compact, simple and simply connected Lie group and let
\begin{equation}
        \label{eq:bce}
        1\to \U(1) \to \widetilde{LG} \stackrel{\pi}{\to} LG \to 1
\end{equation}
be a central extension of the loop group $LG$.
By this we mean a sequence of Lie groups that is exact as a sequence of groups, and where $\pi$ is a principal $\U(1)$-bundle over $LG$. 
By our assumptions on $G$, such central extensions are classified up to \emph{unique} isomorphism by their Chern class  \cite[\S2.2 \& Lemma~2.3.1]{LudewigWaldorf2Group}.
\begin{comment}
(Here an isomorphism of central extensions is a $\U(1)$-equivariant isomorphism of Lie groups preserving the projection to $LG$.)
\end{comment}
A central extension is called \emph{basic} if its Chern class is a generator of $H^2(LG, \Z) \cong \Z$.
In the following we consider the group $G=\Spin (d)$ and a basic central extension
\begin{equation*}
1 \to \U(1) \to \widetilde{L\Spin(d)} \stackrel\pi\to L\Spin(d) \to 1\text{.}
\end{equation*}
For the next definition, but also to be used later, we introduce for paths $\gamma_1, \gamma_2 \in P_e\Spin(d)$ with a common end point, i.e., $\gamma_1(\pi)=\gamma_2(\pi)$, the notation $\gamma_1\cup\gamma_2\in L\Spin(d)$, defined by concatenation of $\gamma_1$ with the reverse of $\gamma_2$:
\begin{equation}
\label{DeltaGamma}
(\gamma_1\cup\gamma_2)(t) = \begin{cases} \gamma_1(t)  & t \in [0, \pi] \\ \gamma_2(2 \pi - t) & t \in [\pi, 2\pi]. 
\end{cases}
\end{equation} 
Moreover, we will use the notation $\Delta\gamma := \gamma\cup\gamma$.

\begin{definition}
\label{DefinitionString2Group}
The \emph{string 2-group} $\XString(d)$ is the crossed module 
\begin{equation}
\label{TMapStringGroup}
 t : \widetilde{L_{[0, \pi]}\Spin(d)} \to P_e\Spin(d),
\end{equation}where $\widetilde{L_{[0, \pi]}\Spin(d)}$ is the restriction of the basic central extension to $L_{[0, \pi]} \Spin(d)$ and $t  := r \circ \pi$ is the composition of the  projection with the restriction map \eqref{InclusionLoopsPaths}.
The crossed module action $\alpha$ is given by
\begin{equation}
\label{CrossedModuleActionStringGroup}
  \alpha(\gamma, \Phi) = \widetilde{\Delta\gamma} \,\Phi\, \widetilde{\Delta\gamma}^{-1}, \qquad \gamma \in P_e\Spin(d), ~~\Phi \in \widetilde{L_{[0, \pi]}\Spin(d)},
\end{equation} 
where $\widetilde{\Delta\gamma}$ is any lift of the loop $\Delta\gamma \in L\Spin(d)$. 
\end{definition}

The action $\alpha$ is well-defined as any two lifts of $\Delta \gamma$ differ by a central element $z \in \U(1)$, and smooth because the central extension has smooth local sections.
Verification of the second identity of \cref{CrossedModuleActions} for this action (the ``Peiffer identity'') uses \cite[Lemma~3.2.2]{LudewigWaldorf2Group} that any central extension of $L\Spin(d)$ has the property of being \emph{disjoint commutative} \cite[Corollary 2.4.4]{LudewigWaldorf2Group}, which means that elements $\Phi, \Psi \in \widetilde{L\Spin(d)}$ commute if $\pi(\Phi)$ and $\pi(\Psi)$ have disjoint supports.

The  groups $\pi_0\XString(d)$ and $\pi_1\XString(d)$ are actually finite-dimensional Lie groups:
\begin{equation*}
\pi_0\XString(d) = \Spin(d)
\quand
\pi_1\XString(d) = \U(1)\text{.}
\end{equation*}
Thus, the string 2-group is an extension
\begin{equation*}
\mathscr{B}\!\U(1) \to \XString(d) \to\Spin(d)_{\dis}\text{.}
\end{equation*}
The fact that the universal central extension is central implies that this extension is central, too.
The condition that it is basic ensures that -- after taking geometric realizations -- $\XString(d)$ a the 3-connected cover of $\Spin(d)$ \cite[Theorem 3.4.2]{LudewigWaldorf2Group}, and hence, by definition, a valid string group.

\begin{remark}
Our model for the string 2-group, \cref{DefinitionString2Group},  depends on the choice of a basic central extension of $L\Spin(d)$, i.e., on the choice of a generator for $H^2(L\Spin(d), \Z)\cong \Z$.
Replacing a generator by its negative, we obtain another crossed module, $\XString(d)^\prime$, which is canonically isomorphic to $\XString(d)$ via a strict intertwiner  that preserves the projection to $\Spin(d)_{\mathrm{dis}}$, but acts as inversion on $\mathscr{B}\U(1)$.
Hence, the two crossed modules $\XString(d)$ and $\XString(d)^\prime$ are isomorphic as crossed modules over $\Spin(d)_{\mathrm{dis}}$, but \emph{not} as $\mathscr{B}\U(1)$-central extensions of $\Spin(d)_{\mathrm{dis}}$.
One way to resolve this sign issue is to insist that the generator of $H^2(L\Spin(d), \Z) \cong H^3(\Spin(d), \Z)$ defining the basic central extension is represented by the left-invariant 2-form $\overline{\sigma}$ corresponding to the Lie algebra cocycle $\sigma \in H^3(\mathfrak{spin}(d), \Z)$ given by
\begin{equation*}
  \sigma(x, y, z) = \langle [x, y], z\rangle
\end{equation*}
for a \emph{positive definite} left-invariant inner product on $\mathfrak{spin}(d)$.
By \cite[Prop.~4.4.4]{PressleySegal}, this corresponds to the requirement that this Chern class is the left-invariant 2-form on $L\Spin(d)$ determined by the Lie algebra cocycle
\begin{equation} \label{BasicCocycle}
  \omega(X, Y) = \frac{1}{2\pi i} \int_{S^1} \langle X(t), Y^\prime(t) \rangle dt
\end{equation}
on $L \mathfrak{spin}(d)$.
\end{remark}

\section{The implementer model}
\label{SectionOperatorModels}

In this section, we describe a specific model for the basic central extension of $L\Spin(d)$ that we will use in the construction of $\XString(d)$ in \cref{DefinitionString2Group}, as well as a closely related model for the hyperfinite type III$_1$ factor on which $\XString(d)$ will be represented.
We write
\begin{equation}
\label{H}
  H := L^2(S^1, \R^d)
\end{equation}
for the space of $\R^d$-valued square-integrable functions on $S^1$ and we let $H^\C$ be its complexification.
Along the projection $L\Spin(d) \to L\SO(d)$, elements $\gamma$ of $L\Spin(d)$ act orthogonally on $H$ by pointwise multiplication. We denote by $\omega(\gamma)$ the corresponding element of $\O(H)$.

We recall that a complex subspace $L \subset H^\C$ is called \emph{Lagrangian} if $\overline{L} = L^\perp$, equivalently, if $L \oplus \overline{L} = H^\C$.
A specific example for such a Lagrangian is the space
\begin{equation*}
  L := \overline{\{ e^{i(n+1/2)t} \mid n =0, 1, 2, \dots\}}.
\end{equation*}
The \emph{restricted orthogonal group} $\O_{\mathrm{res}}(H)$ (with respect to this choice of $L$) is the subgroup consisting of those $g \in \O(H)$ whose  commutator $[g, P_L]$ with the orthogonal projection $P_L$ onto $L$ is a Hilbert-Schmidt operator.
It turns out that the elements $\omega(\gamma) \in \O(H)$, $\gamma \in L\Spin(d)$, are actually contained in $\O_{\res}(H)$ \cite[Lemma 3.22]{Kristel2019};
hence we obtain a group homomorphism
\begin{equation*}
    \omega: L\Spin(d) \to \O_{\mathrm{res}}(H).
\end{equation*}
The group $\O_{\mathrm{res}}(H)$ has the structure of a Banach Lie group for which this homomorphism is smooth, see \cite[\S 3.4, Prop. 3.23]{Kristel2019}.

We will construct a central extension of $L\Spin(d)$ by pulling back a certain central extension of $\O_{\mathrm{res}}(H)$. 
This central extension of $\O_{\mathrm{res}}(H)$ is defined in terms of the Fock representation associated to $L$ for the Clifford algebra of $H$.
To describe this, let $\Cl(H)$ be the algebraic Clifford algebra of $H$, which is the universal complex unital algebra generated  by $H^\C$ and subject to the relation
\begin{equation*}
  v \cdot w + w \cdot v = - 2 \langle \overline{v}, w\rangle, \qquad v, w \in H^\C.
\end{equation*}
The Clifford algebra $\Cl(H)$ carries a unique $*$-operation such that $v^* = - \overline{v}$ for $v \in H^\C$, turning it into a $*$-algebra.
We define the  \emph{Fock space} of $L$ by 
\begin{equation*}
\mathfrak{F} := \overline{\Lambda L} = \overline{\bigoplus_{k=0}^\infty \Lambda^k L}\text{,}
\end{equation*}
 \ie the Hilbert space closure of the algebraic exterior algebra on $L$.
The Fock space carries a canonical $*$-representation 
 \begin{equation}
\label{StarRepresentationCliffordAlgebra}
\pi: \Cl(H) \to \BB(\mathfrak{F})
\end{equation}
of the Clifford algebra $\Cl(H)$, fully determined on the generating set $H \subset \Cl(H)$ by the two properties 
 \begin{equation*}
\begin{aligned}
  \pi(v) \xi &= v \wedge \xi & & v \in L, ~\xi \in \mathfrak{F} \\
  \pi(\overline{v})\xi &= - \pi(\overline{v})^*\xi & & v \in H, ~\xi \in \mathfrak{F}.
\end{aligned}
\end{equation*}
The Fock space $\mathfrak{F}$ has a $\Z_2$-grading given by its decomposition in even and odd degree forms, and the $*$-representation $\pi$ is grading-preserving.
For details, see, e.g., \cite[\S3]{Kristel2019} or \cite[\S3]{LudewigClifford}.

By the universal property of the Clifford algebra, any $g \in \O(H)$ induces an automorphism $\Cl_g$ of $\Cl(H)$, called the \emph{Bogoliubov automorphism} associated to $g$.
An element $g \in \O(H)$ is \emph{implementable}  on $\mathfrak{F}$ if $\Cl_g$ extends along the representation $\pi$ to an automorphism of $\BB(\mathfrak{F})$. 
Since that automorphism is necessarily inner, $g$ is implementable if and only if  there exists $U \in \U(\mathfrak{F})$ (an \emph{implementer for} $g$) such that 
\begin{equation}
\label{ImplementerRelation}
\pi(\Cl_g(a)) = U\pi(a)U^{*}\qquad \forall a\in \Cl(H)\text{.}
\end{equation}
It is a classical result, see \cite[Theorem 3.3.5]{PR95} or \cite[Theorem 6.3]{Ara85}, that $g$ is implementable  if and only if $g$ lies in $\O_{\mathrm{res}}(H)$; moreover, any two implementers for $g$ differ by an element of $\U(1)$. 
In particular, the extension of $\Cl_g$ to an automorphism of $\mathcal{B}(\mathfrak{F})$ is unique, if it exists. We denote that unique extension of $\Cl_g$, for $g\in \O_{\mathrm{res}}(H)$, by $\overline{\Cl}_g$.
The set of implementers, i.e., unitaries satisfying \eqref{ImplementerRelation} for \emph{some} $g \in \O_{\mathrm{res}}(H)$, form a subgroup $\Imp(H) \subset \U(\mathfrak{F})$.
One can show that there exists a Banach Lie group structure on $\Imp(H)$, in such a way that it forms a central extension
\begin{equation}
\label{ImplementerExtensionSeq}
 1\to  \U(1) \to \Imp(H) \to \O_{\res}(H)\to 1
\end{equation}
of Banach Lie groups \cite[\S3.5]{Kristel2019}. 

The groups just discussed fit into a commutative diagram
\begin{equation}
\label{ImplementerExtension}
\begin{aligned}
\xymatrix{
  \Imp(H) \ar@{^(->}[r] \ar[d] & \U(\mathfrak{F}) \ar[d] \\
  \O_{\res}(H) \ar[r] & \PU(\mathfrak{F})
}
\end{aligned}
\end{equation}
of groups, where the bottom horizontal map assigns to an orthogonal transformation $g$ the class of an implementer $U$ satisfying \eqref{ImplementerRelation}.
It follows from \cite[Prop.~3.18]{Kristel2019} that the Lie group topology on $\Imp(H)$ is strictly finer than the strong topology induced from $\U(\mathfrak{F})$; in other words, we obtain that the top horizontal map in \eqref{ImplementerExtension}, i.e., the inclusion, is continuous (for the strong topology on $\U(\mathfrak{F})$).
As the left vertical map admits smooth local sections  and the right vertical map is continuous, this implies that also the bottom horizontal map of \eqref{ImplementerExtension} is continuous.
It follows that \eqref{ImplementerExtension} is in fact a pullback diagram in the category of topological groups.
%
\begin{comment}
It is clearly a set-theoretic pullback.
Now if $\Imp(H)^\prime$ denotes the pullback in the category of topological groups, $\Imp(H)^\prime$ is in particular a central $\U(1)$-extension of $\O_{\res}(H)$, and from the universal property of the pullback, we obtain a continuous map $\alpha : \Imp(H) \to \Imp(H)^\prime$ which is an isomorphism of groups.
After pulling back the topology on $\Imp(H)^\prime$ along $\alpha$, we obtain two topologies on $\Imp(H)$ that turn $\Imp(H)$ into a $\U(1)$-principal bundle, with one finer than the other.
But since both are principal bundles, over small open sets  $O \subset \O_{\res}(H)$, they must both look like the product $O \times \U(1)$, hence the topologies coincide.
\end{comment}
%

\begin{lemma} 
\label{LemmaPullbackBasic}
Let $d\geq 5$. 
The central extension of $L\Spin(d)$, 
\begin{equation}
\label{PullbackExtension}
\begin{aligned}
  \xymatrix{
    \widetilde{L\Spin(d)} \ar@{-->}[r] \ar@{-->}[d] & \Imp(H) \ar[d] \ar[r] & \U(\mathfrak{F}) \ar[d]  \\
    L\Spin(d) \ar[r]_{\omega} & \O_{\res}(H) \ar[r] & \PU(\mathfrak{F})\text{,}
  }
\end{aligned}
\end{equation}
defined as the pullback of the central extension \eqref{ImplementerExtensionSeq} along $\omega$ is basic.
\end{lemma}

\begin{proof}
By the assumption $d \geq 5$, the homomorphism $\omega$ induces an isomorphism on $H^2$ (this follows, e.g., from \cite[Prop.~12.5.2]{PressleySegal}) and the Chern class of $\Imp(H)$ is classically known to be a generator of $H^2(\O_{\mathrm{res}}(H), \Z) \cong \Z$  \cite[Prop.~1.2 (iv)]{SW}.
This implies the statement.
In fact, one can show (see \cite[Thm.~3.26]{Kristel2019}) that the Lie algebra cocycle defining the extension is precisely the cocycle \eqref{BasicCocycle}.
\end{proof}

We call the central extension of \cref{LemmaPullbackBasic} the \emph{implementer model} for the basic central extension of $L\Spin(d)$. 
It will be essential for our construction of the stringor representation in \cref{sec:theStringorRep} that we use the implementer model in the definition of the string 2-group $\XString(d)$  (see \cref{DefinitionString2Group}).

Next we define the von Neumann algebra $A$ on which we represent $\XString(d)$.
We write
\begin{equation*}
\begin{aligned}
   H_0 := \{f \in H \mid \mathrm{supp}(f) \subset [0, \pi]\}\subset H
\end{aligned} 
\end{equation*}
for the Hilbert space of functions with support in $[0, \pi]$.
Its algebraic Clifford algebra $\Cl(H_0)$ is a subalgebra of the algebraic Clifford algebra $\Cl(H)$.
Restricting the representation \eqref{StarRepresentationCliffordAlgebra} to this subalgebra, we define the von Neumann algebra
\begin{equation}
\label{OurRepresentationSpace}
  A := \pi(\Cl(H_0))^{\prime\prime} \subset \BB(\mathfrak{F}),
\end{equation}
\ie the bicommutant in $\BB(\mathfrak{F})$.
Equivalently, by von Neumann's bicommutant theorem, $A$ is the closure of $\pi(\Cl(H_0))$ in either the weak or strong operator topology.

We also consider the orthogonal complement $H_0^{\perp}= \{f \in H \mid \mathrm{supp}(f) \subset [\pi, 2\pi]\}$ of $H_0$, and the corresponding von Neumann algebra
\begin{equation*}
  A_\perp := \pi(\Cl(H_0^\perp))^{\prime\prime} \subset \BB(\mathfrak{F})\text{.}
\end{equation*}
Since the subalgebras $\Cl(H_0^\perp)$ and $\Cl(H_0)$ of $\Cl(H)$ super commute, it follows that also the von Neumann subalgebras $A$ and $A_\perp$ of $\BB(\mathfrak{F})$ super commute. 
However, we have the following, more precise result, called the \emph{twisted duality} property of the Clifford von Neumann algebra.

\begin{theorem}[{{\cite{BLJ02}}}]
\label{ThmTwistedDuality}
$A$ and $A_\perp$ are each other's super commutant. 
In other words, we have
\begin{equation*}
 A_\perp = \bigl\{a \in \BB(\mathfrak{F}) \mid \forall b \in A: ab = (-1)^{|a||b|} ba\bigr\}
\end{equation*}
and the same equation with the roles of $A$ and $A_\perp$ exchanged.
\end{theorem}

\begin{remark}
For each closed subset $I \subset S^1$, we may denote by $H_I \subset H$ the Hilbert subspace of functions supported in $I$ and obtain corresponding Clifford algebra $\Cl(H_I)$ with von Neumann completions $A_I$.
It turns out that the assignment $I \mapsto A_I$ has a certain equivariance property with respect to the action of the M\"obius group  on the circle, which constitutes the structure of a \emph{conformal net}. 
It is well-known that the von Neumann algebras $A_I$ (in particular the algebra $A$ defined in \eqref{OurRepresentationSpace}) for proper intervals $I \subset S^1$ are of type III$_1$, see e.g.~\cite[Lemma 2.9]{Gabbiani1993}. In particular, all these algebras are isomorphic.
The von Neumann algebra $A$ is also hyperfinite, as the union of the subalgebras $\Cl(W_n)$ is ultraweakly dense, where $W_1 \subset W_2 \subset \dots$ is an ascending sequence of finite-dimensional subspaces the union of which is dense in $H_0$.
\end{remark}

Summarizing, we obtain the following diagram of subalgebras of $\mathcal{B}(\mathfrak{F})$:
\begin{equation}
\label{diagram-of-subalgebras}
\alxydim{}{\pi(\Cl(H_0)) \ar[rr]\ar[dr] && A\ar[dr]  \\ & \pi(\Cl(H)) \ar[rr] && \mathcal{B}(\mathfrak{F}) \\   \pi(\Cl(H_0^{\perp})) \ar[rr] \ar[ur] && A_{\perp} \ar[ur] }
\end{equation}

\section{The stringor representation}
\label{sec:theStringorRep}

In this section we construct the stringor representation of the string 2-group $\XString(d)$, for $d \geq 5$.
Accordant with \cref{continuous-unitary-representation}, it will be a continuous strict intertwiner
\begin{equation}
\label{StringRep}
        \mathscr{R}: \XString(d) \to \UAUT(A),
\end{equation} 
where $\XString(d)$ is the the string 2-group defined in \cref{DefinitionString2Group}, $A$ is the von Neumann algebra defined in \eqref{OurRepresentationSpace}, and $\UAUT(A)$ is its unitary automorphism 2-group of $A$ as in \cref{DefCrossedModuleAUT}.

We recall that the implementer model for the basic central extension $\widetilde{L\Spin(d)}$, set up in \cref{SectionOperatorModels},  comes with a continuous group homomorphism 
\begin{equation}
\label{GroupHomToUfrak}
\Omega: \widetilde{L\Spin(d)} \to \U(\mathfrak{F})
\end{equation}
with the property that if $\Phi \in \widetilde{L\Spin(d)}$ projects to a loop $\gamma\in L\Spin(d)$, then $\Omega(\Phi)$ implements $\omega(\gamma)\in \O_{\res}(H)$, i.e., the extended Bogoliubov automorphism $\overline{\Cl}_{\omega(\gamma)}$ of $\mathcal{B}(\mathfrak{F})$ is conjugation with $\Omega(\Phi)$.

\begin{lemma}
\label{LemmaNA}
  The group homomorphism $\Omega$ from \cref{GroupHomToUfrak} takes values in the subgroup
\begin{equation*}
  N(A) = \bigl\{ U \in \U(\mathfrak{F}) \mid \forall a \in A : UaU^{*} \in A\bigr\}.
  \end{equation*}
\end{lemma}

\begin{proof}
For any $\gamma \in L\Spin(d)$, the  orthogonal transformation $\omega(\gamma)$ of $H$ preserves the subspace $H_0$.
Correspondingly, the  Bogoliubov automorphism $\Cl_{\omega(\gamma)}$ preserves the subalgebra $\Cl(H_0) \subset \Cl(H)$.
Since any unitary $U\in \U(\mathfrak{F})$  in the image of $\Omega$ implements some $\omega(\gamma)$ via conjugation,  conjugation with $U$ preserves $\pi(\Cl(H_0))$ and consequently its weak closure $A$.
\end{proof}

\begin{remark}
\label{normaliser-subgroup-is-closed}
We remark that
\begin{equation*}
N(A) =  \bigcap_{a \in A} \bigcap_{b \in A^\prime} \{ U \in \U(\mathfrak{F}) \mid UaU^*b - bUaU^* = 0 \},
\end{equation*}
which shows that $N(A)$ is a closed subgroup of $\U(\mathfrak{F})$.
\end{remark}

By \cref{LemmaNA}
the group homomorphism $\Omega$ restricts to a continuous group homomorphism 
\begin{equation}
\label{restriction-to-normaliser}
\Omega^\prime :\widetilde{L\Spin(d)} \to N(A)\text{.}
%\tilde \Rep_1:\widetilde{L\Spin(d)} \to N(A)\text{.}
\end{equation}
By definition of $N(A)$, its elements $U\in N(A)$ act by conjugation on $A$, thus defining a group homomorphism 
\begin{equation}
\label{from-normalizer-to-automorphisms}
t_{A}:N(A) \to \Aut^{*}(A)\text{.}
\end{equation}
The map $t_{A}$ is continuous by \cite[Lemma A.18]{ConformalNetsI}; this uses that the canonical vacuum vector $\Omega \in \mathfrak{F}$ is  cyclic and separating  for the action of $A$ on $\mathfrak{F}$, thus turning $\mathfrak{F}$ into a standard form for $A$. 
\begin{lemma}
\label{DiagramNA}
For any $\gamma\in L\Spin(d)$, the extended Bogoliubov automorphism $\overline{\Cl}_{\omega(\gamma)}$ of $\mathcal{B}(\mathfrak{F})$ preserves the subalgebra $A\subset \mathcal{B}(\mathfrak{F})$ and hence restricts to an automorphism $\overline{\Cl}{}^\prime_{\omega(\gamma)} \in \Aut(A)$. The corresponding map $\overline{\Cl}{}^\prime_\omega: L\Spin(d) \to \Aut(A)$ is a continuous group homomorphism, and the following diagram is commutative:
\begin{equation*}
\begin{aligned}
  \xymatrix{
    \widetilde{L \Spin(d)} \ar[r]^-{\Omega^\prime} \ar[d] & N(A) \ar[d]^{t_{A}}  \\
    {L\Spin(d)} \ar[r]_-{\overline{\Cl}{}^\prime_\omega} & \Aut(A)\text{.}
  }
\end{aligned}
%\begin{aligned}
%  \xymatrix{
%    \widetilde{L \Spin(d)} \ar[r]^-{\tilde \Rep_1} \ar[d] & N(A) \ar[d]^{t_{A}}  \\
%    {L\Spin(d)} \ar[r]_-{\tilde R_0} & \Aut(A)\text{.}
%  }
%\end{aligned}
\end{equation*}
Moreover, the automorphism  $\overline{\Cl}{}^\prime_{\omega(\gamma)}$ only depends on the restriction of $\gamma$ to its first half, i.e., if $\gamma,\gamma',\gamma''\in P\Spin(d)$ have a common initial point and a common end point, then 
\begin{equation*}
\overline{\Cl}{}^\prime_{\omega(\gamma\cup\gamma')}= \overline{\Cl}{}^\prime_{\omega(\gamma\cup\gamma'')}\text{.}
%\tilde R_0(\gamma\cup\gamma')=\tilde R_0(\gamma\cup\gamma'')\text{.}
\end{equation*}  
\end{lemma}

\begin{proof}
If $\Phi \in \widetilde{L\Spin(d)}$ is any lift of $\gamma$, then $U:=\Omega'(\Phi)\in N(A)$ implements $\omega(\gamma)$. Thus, $t_{A}(U)$ is the claimed restriction  $\overline{\Cl}{}^\prime_{\omega(\gamma)}$, and the diagram is commutative. 
Continuity of  $\overline{\Cl}{}^\prime_\omega$ follows from the fact that the basic central extension admits local sections, and that  $\Omega^\prime$ and $t_{A}$ are continuous. 
In the situation of the three paths $\gamma,\gamma',\gamma''$, the orthogonal transformations $\omega(\gamma\cup\gamma')$ and $\omega(\gamma\cup\gamma'')$ restrict to the same transformation on $H_0$ and hence induce the \emph{same} Bogoliubov automorphisms on $\Cl(H_0)$. But then, the restrictions $\overline{\Cl}{}^\prime_{\omega(\gamma\cup\gamma')}$ and $\overline{\Cl}{}^\prime_{\omega(\gamma\cup\gamma'')}$ of the extended Bogoliubov automorphisms $\overline{\Cl}_{\omega(\gamma\cup \gamma')}$ and $\overline{\Cl}_{\omega(\gamma\cup \gamma'')}$, respectively, to the weak closure $A$ of $\Cl(H_0)$ also coincide, as $*$-automorphisms of von Neumann algebras are automatically weakly continuous.
\end{proof}

Next we note that $N(A)$ contains the subgroup $\U(A) \subset \U(\mathfrak{F})$ of unitary elements in $A$.
We have the following lemma.

\begin{lemma}
\label{LemmaUinNA}
The restriction of the group homomorphism  $\Omega'$ to the subgroup $\widetilde{L_{[0, \pi]} \Spin(d)}$ factors through the subgroup $\U(A)$:
\begin{equation*}
%\label{PullbackExtension}
\begin{aligned}
  \xymatrix{
    \widetilde{L_{[0, \pi]} \Spin(d)} \ar@{-->}[r] \ar@{^{(}->}[d] & \U(A) \ar@{^{(}->}[d]  \\
    \widetilde{L\Spin(d)} \ar[r]_-{\Omega^\prime} & N(A)
  }
\end{aligned}
\end{equation*}
\end{lemma}

\begin{proof}
Let $U$ be an implementer for  $\omega(\gamma)$, where $\gamma \in L_{[0, \pi]}\Spin(d)$.
We first note that $U$ is a \emph{grading preserving} unitary transformation of $\mathfrak{F}$.
Indeed, since $L\Spin(d)$ is connected, $\omega(\gamma)$ is contained in the identity component of $\O_{\res}(H)$ and elements in the identity component are implemented by grading preserving unitaries \cite[\S3.5 \& p.~134]{PR95}.
Now, as $\gamma$ is supported on $[0, \pi]$, $\omega(\gamma)$ acts trivially on the subspace $H_0^\perp \subset H$, so $\Cl_{\omega(\gamma)}$ acts trivially on $\Cl(H_0^\perp)$.
Hence, $U$ is contained in the commutant of $\Cl(H_0^\perp)$.
Because $U$ is even and the even parts of the commutant and of the super commutant coincide, this implies that $U$ is also contained in the super commutant of $\Cl(H_0^\perp)$.
But this is the same as the super commutant of its weak closure $A_\perp$, which is $A$ by \cref{ThmTwistedDuality}.
\end{proof}

We obtain from \cref{LemmaUinNA} a continuous group homomorphism 
\begin{equation*}
R_1: \widetilde{L_{[0, \pi]} \Spin(d)} \to \U(A)\text{,}
\end{equation*}
fitting into the following commutative diagram of topological groups:
\begin{equation}
\label{big-commutative-diagram}
\alxydim{}{\widetilde{L_{[0, \pi]} \Spin(d)} \ar[r]^-{R_1} \ar@{^(->}[d] & \U(A) \ar@{^(->}[d] \\ \widetilde{L\Spin(d)} \ar[r]^-{\Omega^\prime} \ar[d] & N(A) \ar[d]^{t_{A}} \\ L\Spin(d) \ar[r]_{\overline{\Cl}{}^\prime_\omega} & \Aut(A)\text{.}}
\end{equation}
Finally, we define the continuous group homomorphism 
\begin{equation}
\label{definition-of-R0}
R_0 := \overline{\Cl}{}^\prime_\omega \circ \Delta: P_e\Spin(d) \to \Aut(A)
%R_0 := \tilde R_0 \circ \Delta: P_e\Spin(d) \to \Aut(A)
\end{equation}
using the path doubling map \cref{DeltaGamma}.
Now we have set up the required structure for a strict intertwiner $\mathscr{R}:=(R_0,R_1):\XString(d) \to \UAUT(A)$, and we are in position to prove the main result of this article.

\begin{theorem}
The pair $\mathscr{R}=(R_0,R_1)$ is a unitary representation  $\mathscr{R}: \XString(d) \to \mathscr{U}(A)$.
\end{theorem}

\begin{proof}
First we have to check compatibility with the maps $t$ of the two crossed modules, which we will denote for the moment by $t_{\XString(d)}$ and $t_{\UAUT(A)}$, respectively. We notice that $t_{\UAUT(A)} =t_{A}|_{\U(A)}$.  Thus, we have to show that 
\begin{equation}
\label{target-matching-condition}
R_0(t_{\XString(d)}(\Phi))=t_{A}(R_1(\Phi))
\end{equation}
holds for all $\Phi \in \widetilde{L_{[0, \pi]} \Spin(d)}$. Let $\gamma := t_{\XString(d)}(\Phi)\in P_e\Spin(d)$, i.e.,  $\Phi$ projects to the loop $\gamma\cup \mathrm{c}_e$. By commutativity of the diagram \cref{big-commutative-diagram}, we have
$t_{A}(R_1(\Phi))=\overline{\Cl}{}^\prime_{\omega(\gamma\cup \mathrm{c}_e)}$.
On the other hand, we have 
\begin{equation*}
R_0(t_{\XString(d)}(\Phi)) = R_0(\gamma)=\overline{\Cl}{}^\prime_{\omega(\gamma\cup \gamma)}\text{.}
\end{equation*} 
By \cref{DiagramNA} we have $\overline{\Cl}{}^\prime_{\omega(\gamma\cup \mathrm{c}_e)} = \overline{\Cl}{}^\prime_{\omega(\gamma\cup \gamma)}$; this shows \cref{target-matching-condition}.
 
 Second, we have to check that $\mathscr{\Rep}$ intertwines the crossed module actions, which we denote for the moment by $\alpha_{\XString(d)}$ and $\alpha_{\UAUT(A)}$, respectively.
 Let $\gamma \in P_e \Spin(d)$ and $\Phi \in \widetilde{L_{[0, \pi]} \Spin(d)}$; moreover, let $\widetilde{\Delta\gamma}\in \widetilde{L\Spin(d)}$ be a lift of the loop $\Delta\gamma$. Then, we compute
 \begin{equation*}
 \begin{aligned}
   R_1 (\alpha_{\XString(d)}(\gamma, \Phi)) &= R_1(\widetilde{\Delta\gamma} \;\Phi\; \widetilde{\Delta\gamma}^{-1}) && \text{Definition of $\alpha_{\XString(d)}$, see \cref{CrossedModuleActionStringGroup}}  \\
   &= \Omega^\prime(\widetilde{\Delta\gamma})\;R_1( \Phi)\;\Omega^\prime( \widetilde{\Delta\gamma})^{-1} && \text{Definition of $R_1$, see \cref{LemmaUinNA}} \\
%   &= \tilde R_1(\widetilde{\Delta\gamma})\;R_1( \Phi)\;\tilde R_1( \widetilde{\Delta\gamma})^{-1} && \text{Definition of $R_1$, see \cref{LemmaUinNA}} \\
   &= t_{A}(\Omega^\prime(\widetilde{\Delta\gamma}))(R_1(\Phi)) && \text{Definition of $t_{A}$, see \cref{DefCrossedModuleAUT}} \\
%   &= t_{A}(\tilde R_1(\widetilde{\Delta\gamma}))(R_1(\Phi)) && \text{Definition of $t_{A}$} \\
   &= \overline{\Cl}{}^\prime_{\omega(\Delta \gamma)}(R_1(\Phi)) && \text{Definition of $\overline{\Cl}{}^\prime_{\omega}$, see \cref{DiagramNA}} \\
%   &= \tilde R_0(\Delta\gamma)(R_1(\Phi)) && \text{Definition of $\tilde R_0$, see \cref{DiagramNA}} \\
   &=  R_0(\gamma)(R_1(\Phi)) && \text{Definition  of $R_0$, see \cref{definition-of-R0}} \\
   &= \alpha_{\UAUT(A)}(R_0(\gamma), R_1( \Phi)) && \text{Definition of $\alpha_{\UAUT(A)}$, see \cref{DefCrossedModuleAUT}.}
 \end{aligned}
 \end{equation*}
 This completes the proof. 
\end{proof}

We conclude our construction of the stringor representation $\mathscr{R}$ by several remarks. 

\begin{remark}
\label{RemarkvNnecessary}
  The  operator $\Omega(\Phi)\in \U(\mathfrak{F}) \subset \mathrm{B}(\mathfrak{F})$ associated to $\Phi \in \widetilde{L_{(0,\pi)}\Spin}(d)$ is generally not contained in $\pi(\Cl(H_0))\subset \BB(\mathfrak{F})$ or its norm completion. 
  Therefore, $\XString(d)$ is \emph{not} represented on the algebraic Clifford algebra $\Cl(H_0)$ or its  $\mathrm{C}^*$-completion, and the passage to von Neumann algebras is inevitable.
\end{remark}

\begin{remark}
\label{RemarkRepNontrivial}
Since $A$ is a factor, we have $\pi_1\AUT(A) =\mathrm{ker}(t_{\UAUT(A)})= \U(A) \cap  A' = \U(1)$. 
On $\pi_1$, 
the stringor representation induces the identity, $\pi_1\mathscr{R}=\id_{\U(1)}$. 
On $\pi_0$,
it induces a continuous map $\pi_0 \mathscr{\Rep}: \Spin(d) \to \pi_0\AUT(A) = \Out^{*}(A)$. 
\begin{comment}
Hence, the k-invariant $k_{\AUT(A)}$ of $\AUT(A)$ (see \cref{sec:2groups}) satisfies
        \begin{equation*}
                (\pi_0\Rep)^{*}k_{\AUT(A)} = k_{\XString(d)},
        \end{equation*}
        which is a generator of $\h^3_{\mathrm{SM}}(\Spin(d), \U(1)) \cong \h^4(B\Spin(d),\Z)$.
        This shows that the k-invariant $k_{\AUT(A)}$ is prime in the abelian group $\h^3_{\mathrm{SM}}(\Out(A),\U(1))$.   
%        (Because if it would be divisible, i.e., $k_{\AUT(A)}=m\cdot k'$, then $\pm 1=k_{\XString(d)}=m \cdot (\pi_0\Rep)^{*}k' \in \Z$, so that we get $m=\pm 1$.)
In that sense, the extension 
\begin{equation*}
\mathscr{B} \U(1) \to \AUT(A) \to \pi_0\AUT(A)_{\mathrm{dis}} 
\end{equation*}
is \quot{basic}.
It also shows that our representation is non-trivial, in the sense that it is not naturally isomorphic to a constant representation. 
\end{comment}
\end{remark}

\begin{remark}
\label{re:strict}
It is no surprise that the stringor representation can be realized as a \emph{strict} intertwiner, as opposed to a weak morphism of topological crossed modules, also known as  butterfly \cite{Aldrovandi2009}. Indeed, any  butterfly between $\XString(d)$ and a topological crossed module $\mathscr{G}$,
\begin{equation*}
\xymatrix@C=3em{\mquad\widetilde{L_{[0, \pi]}\Spin(d)}\mquad \ar[dd] \ar[dr] && H \ar[dd] \ar[dl] \\ & B \ar[dr]\ar[dl] \\ \mquad P_e\Spin(d)\mquad \ar@{-->}@/^1pc/[ur] && G}
\end{equation*}
has a section against its NE-SW-sequence, since that sequence is a short exact sequence of topological groups, and  $P_e\Spin(d)$ is  contractible as a topological group.
Hence, the given butterfly  is isomorphic to a \emph{strict morphism} \cite[Prop. 4.5.3]{Aldrovandi2009}, corresponding to a \emph{crossed intertwiner} \cite{Nikolause}, or -- in terms of monoidal groupoids -- a continuous monoidal functor with continuous associator.
\end{remark}

\begin{remark}
\label{smooth-representation}
By construction, the stringor representation is continuous. We recall that the Fr\'echet Lie groups appearing in the crossed module $\XString(d)$ have been equipped with their underlying (Fr\'echet) topology. The forgetful functor $\frech \to \top$ factors through the category $\diff$ of diffeological spaces via  the \quot{manifold diffeology} functor $M: \frech \to \diff$ and the \quot{D-topology} functor $D: \diff \to \top$. The latter has a right adjoint, the functor $C:\top \to \diff$
that equips a topological space with the \quot{continuous diffeology}. As a right adjoint, $C$ preserves limits and hence sends topological crossed modules to diffeological crossed modules. In particular, there is a diffeological version
\begin{equation*}
\UAUT(A)_{\diff} := C(\UAUT(A))
\end{equation*}
of the unitary automorphism 2-group of a von Neumann algebra. By adjunction, our stringor representation induces a \emph{smooth} strict intertwiner
\begin{equation*}
\mathscr{R}^{\infty}: M(\XString(d)) \to \UAUT(A)_{\diff}\text{,}
\end{equation*}
and hence a smooth, diffeological version of the stringor representation. 
\begin{comment}
Indeed, we have continuous maps
\begin{align*}
D(M(P_e\Spin(d)))=(P_e\Spin(d))_{\top} \stackrel{R_0}{\to} \Aut(A)
\\
D(M(\widetilde{L_{[0,\pi]}\Spin(d)})) =(\widetilde{L_{[0,\pi]}\Spin(d)})_{\top} \stackrel{R_1}\to \U(A)   
\end{align*}
By adjunction, they become smooth maps.
\end{comment}
The unitary automorphism 2-group $\UAUT(A)$ of a von Neumann algebra is in fact  delta-generated, meaning that the topologies of  both $\U(A)$ and $\mathrm{Aut}^{*}(A)$ are delta-generated (i.e., determined by the continuous curves). 
This means that $D(\UAUT(A)_{\diff})=\UAUT(A)$, see \cite[Prop. 2.10]{Kihara2018}. Thus, we have $D(\mathscr{R}^{\infty})=\mathscr{R}$, i.e., our stringor representation is obtained by applying the D-topology functor to its smooth version $\mathscr{R}^{\infty}$. 
\end{remark}

\begin{comment}
\begin{lemma}
\label{AUTAIsDeltaGenerated}
$\AUT(A)$ is a $\Delta$-generated topological 2-group.
\end{lemma}

\begin{proof}
We have to show that the spaces $N(A)$ and $\Aut(A)$ are $\Delta$-generated.
As polish groups, they are first-countable and, by \cite{UngarAllKinds}, locally path-connected (see also \cite[Thm.~8]{WhittingtonConnectedness}).
But any first-countable, locally path-connected space $X$ is $\Delta$-generated.
\\
To show this, it suffices to verify that a map $f: X \to Y$ to an arbitrary topological space $Y$ is continuous if for any continuous map $c: [0, 1] \to X$, the composite $f \circ c$ is continuous.
Let $(x_n)_{n \in \N}$ be a sequence in $X$ that converging against some point $x$ and let $U_n$ be a basis of path connected neighborhoods of $x$ such that $x_m \in U_n$ whenever $m \geq n$.
Let moreover $c_n : [\frac{1}{n+1}, \frac{1}{n}] \to U_n$ be a continuous path such that $c_n(\frac{1}{n+1}) = x_{n+1}$ and $c_n(\frac{1}{n}) = x_n$, which exists by path-connectedness of $U_n$, and let $c: (0, 1] \to X$ be the concatenation of these paths.
By construction, setting $c(0) = x$ gives a continuous extension of $c$ to $[0, 1]$. 
Since $f \circ c$ is continuous by assumption, $f(x_n) = (f \circ c)(\frac{1}{n})$ converges to $f(x) = (f \circ c)(0)$ as $n \to \infty$. 
This shows that $f$ is sequentially continuous, which implies continuity by first-countability of $X$.
\end{proof}

\end{comment}\end{remark}

\appendix

%\setsecnumdepth{2}

\section{The 2-group perspective}

\label{2-group-perspective}

In this appendix, we describe another perspective to our stringor representation, namely from the point of view that strict 2-groups are groupoids internal to the category of  groups. Therefore, in this section, we distinguish intentionally between crossed modules (used before as a model for strict 2-groups) and the actual  2-groups. Just like crossed modules, 2-groups exist, in particular, in a continuous setting (\quot{topological 2-group}) and in a smooth setting (\quot{Lie 2-group}).   

\subsection*{Strict 2-groups versus crossed modules}

If $\Gamma$ is a topological 2-group, we denote by $\Gamma_1$ and $\Gamma_0$ its topological groups of morphisms and objects, respectively. We further denote by $s,t: \Gamma_1\to\Gamma_0$ the source and target map, respectively, and by $i:\Gamma_0\to \Gamma_1$ the identity-assigning map. 
When constructing   topological 2-groups  it is worthwhile to notice that composition and inversion of the underlying groupoid  are already determined by the remaining structure. Indeed, it is straightforward to see that 
\begin{equation} \label{FormulaForComposition}
x \circ y = x \, i(s(x))^{-1} y = x \, i(t(y))^{-1} y \text{,}
\end{equation}
for composable morphisms $x, y \in \Gamma_1$, i.e., morphisms such that $s(x) = t(y)$.
It follows from this that the inverse of a morphism $x \in \Gamma_1$ with respect to composition satisfies
\begin{equation}
\label{FormulaForInversion}
        \inv(x)=i(s(x))x^{-1}i(t(x))\text{.}
\end{equation}
\begin{comment}
        Indeed,
        \begin{align*}
                \inv(x) \circ x &= i(s(x))x^{-1}i(t(x)) i(s(x^{\circ-1}))^{-1} x\\
                                &= i(s(x))x^{-1}i(t(x)) i(t(x))^{-1} x \\
                                &= i(s(x))
        \end{align*}
\end{comment}
Moreover, the subgroups $\ker(s)$ and $\ker(t)$ of $\Gamma_1$ commute: let $x \in \ker(s)$, $y \in \ker(t)$, and let $e\in \Gamma_0$ be the unit element. Then 
\begin{equation} \label{eq:KerSKerT}
        y x  = (e \circ y) (x \circ e) = (e \cdot x) \circ (y \cdot e) = x \circ y = x \, i(s(x))^{-1} y = xy.
\end{equation}
We have the following simple converse of these three observations.

\begin{lemma}
        \label{LemmaMinimalData2Group}
        Suppose $\Gamma_0$ and $\Gamma_1$ are topological groups and $s, t:\Gamma_1 \to \Gamma_0$ and $i: \Gamma_0\to\Gamma_1$ are  continuous group homomorphisms such that:
        \begin{enumerate}[{\normalfont (a)}]

                \item 
                        \label{cor:MinimalData2Group:a}
                        $s \circ i = \id_{\Gamma_0} = t \circ i$.

                \item
                        \label{cor:MinimalData2Group:b}
                        $\mathrm{ker}(s)$ and $\mathrm{ker}(t)$ are commuting subgroups.

        \end{enumerate} 
        Then, together with the composition defined by \cref{FormulaForComposition} and the inversion defined in \cref{FormulaForInversion}, this structure constitutes a  topological 2-group. 
\end{lemma}

\begin{remark}
\Cref{LemmaMinimalData2Group} holds  verbatim in the smooth case, with the only modification that the subgroups in (b) have to be \emph{Lie} subgroups, see \cite[Lemma 3.3.1]{LudewigWaldorf2Group} for a discussion of the smooth case.  
\end{remark}

Let us recall  the usual equivalence between the category of topological crossed modules and continuous strict intertwiners on one side, and topological 2-groups and 2-group homomorphisms (functors whose component maps are continuous group homomorphisms) on the other side:
\begin{equation} 
\label{AdjunctionLieCross}
\xymatrix{
  \mathscr{X} : \bigset{5.1em}{Topological 2-groups} \ar@<3pt>[r] & \ar@<3pt>[l] \bigset{7.1em}{Topological crossed modules} : \mathcal{G}\text{.}
}
\end{equation}
If $\Gamma$ is a topological 2-group, then we put $G:= \Gamma_0$, $H:=\mathrm{ker}(s)\subset \Gamma_1$, $t:=(t:\Gamma_1\to \Gamma_0)|_{\mathrm{ker}(s)}$, and $\alpha_g(h) := i(g) h\, i(g)^{-1}$ to obtain a topological crossed module $\mathscr{X}(\Gamma)$. Conversely, if $\mathscr{G}=(G,H,t,\alpha)$ is a topological crossed module, then putting $\Gamma_0:= G$, $\Gamma_1 := H \rtimes_{\alpha} G$, $s_{\Gamma}(h,g) := g$, $t_{\Gamma}(h,g) := t(h)g$, and $i(g):= (1,g)$ provides the input data for \cref{LemmaMinimalData2Group} and hence a topological 2-group $\mathcal{G}(\mathscr{G})$.

Moreover, to a strict intertwiner $\mathscr{R}=(R_0,R_1)$ between topological crossed modules, the functor $\mathcal{G}$ associates the 2-group homomorphism $\mathcal{G}(\mathscr{R})$ that is $\mathcal{G}(\mathscr{R})_0:= R_0$ on the level of objects and $\mathcal{G}(\mathscr{R})_1 :=R_1 \times R_0$ on the level of morphisms. 

For instance, our stringor representation $\mathscr{R}$ becomes a   2-group homomorphism
\begin{equation*}
\mathcal{G}(\mathscr{R}): \mathcal{G}(\XString(d)) \to \mathcal{G}(\UAUT(A))\text{.}
\end{equation*} 
The point of this appendix is that the 2-groups $\mathcal{G}(\XString(d))$ and $\mathcal{G}(\UAUT(A))$, as well as the 2-group homomorphism $\mathcal{G}(\mathscr{R})$, have \quot{nicer} descriptions than the ones produced from their crossed module description. More precisely, the groups of morphisms, 
\begin{align*}
\mathcal{G}(\XString(d))_1 &= \widetilde{L_{[0,\pi]}\Spin(d)} \rtimes P_e\Spin(d)
\\
\mathcal{G}(\UAUT(A))_1 &= \U(A) \rtimes \Aut(A)\text{,}
\end{align*}
 have descriptions that use no semi-direct products, see \cref{2-group-version-of-string,2-group-version-of-UAUT}. 
 
 \subsection*{The string group as a strict 2-group}
 
 The case of the string 2-group is treated in \cite[\S 3.3]{LudewigWaldorf2Group}, and we recall this briefly.
The basis is the unique existence of a \emph{fusion factorization} for the basic central extension of the loop group $L\Spin(d)$: a Lie group homomorphism 
\begin{equation}
\label{DefinitionFusionFactorization}
i:P_e\Spin (d) \to \widetilde{L\Spin(d)}
\end{equation} such that $i(\gamma)$ projects to the loop $\Delta\gamma=\gamma \cup\gamma$. An explicit construction of this fusion factorization in terms of the implementer model was given before in \cite[\S 5]{Kristel2019}. 

We denote by $\widetilde{P_e\Spin(d)^{[2]}}$ the pullback of the basic central extension along the map $\cup: P_e\Spin(d)^{[2]} \to L\Spin(d)$. Then,
we set up a 2-group with source and target maps
\begin{equation*}
\alxydim{@R=0em}{\\\widetilde{P_e\Spin(d)^{[2]}} \ar@/^1.5pc/[rr]^-{s} \ar@/_1.5pc/[rr]_-{t} \ar[r] &P_e\Spin(d)^{[2]}  \ar@<3pt>[r]^-{\pr_2} \ar@<-3pt>[r]_-{\pr_1} & P_e\Spin(d)\text{,}\\\strut}
\end{equation*} 
and identity map $i$.
Via \cref{LemmaMinimalData2Group} (using and the disjoint-commutativity of $\widetilde{L\Spin(d)}$, see \cref{sec:string2group}) this defines a Lie 2-group $\GString(d)$.  

\begin{proposition}
The Lie 2-group $\GString(d)$ and the Lie crossed module $\XString(d)$ correspond to each other under the equivalence \cref{AdjunctionLieCross}; precisely, we have
$\mathscr{X}(\GString(d)) =  \XString(d)$.
\end{proposition} 

\begin{proof}
This is clear from the given constructions, and the fact that $i(\gamma)$ can serve as the lift $\widetilde{\Delta\gamma}$ used in \cref{CrossedModuleActionStringGroup}.
\end{proof}

\begin{remark}
\label{2-group-version-of-string}
Corresponding to the equality $\mathscr{X}(\GString(d)) =  \XString(d)$ we have a canonical Fr\'echet Lie 2-group isomorphism
\begin{equation*}
\GString(d) \cong \mathcal{G}(\XString(d))
\end{equation*}
coming from the canonical natural isomorphism $\mathcal{G}\circ \mathscr{X}\cong \id$ belonging to the equivalence \cref{AdjunctionLieCross}. 
Explicitly, on the level of morphisms, the map
\begin{equation*}
\widetilde{P_e\Spin(d)^{[2]}}\to \widetilde{L_{[0,\pi]}\Spin(d)} \rtimes P_e\Spin(d); \quad U \mapsto(U\cdot  i(s(U))^{-1}, s(U)) 
\end{equation*}
is an isomorphism of Fr\'echet Lie groups. 
\end{remark}

\subsection*{The automorphism group of a von Neumann algebra as a strict  2-group}

Next we describe the crossed module $\UAUT(A)$ for the unitary automorphism 2-group of a von Neumann algebra $A$ as a strict topological 2-group.  For this purpose, we choose a standard form $\mathfrak{F}$ of $A$ (e.g., $\mathfrak{F}$ could be defined using a cyclic and separating vector $\Omega \in \mathfrak{F}$).   In this context, we also make  use of the modular conjugation operator $J: \mathfrak{F} \to \mathfrak{F}$ provided by the structure of the standard form.

 We define topological 2-group $\mathcal{U}(A)$ in the following way.
Its group of objects is the  group $\Aut^{*}(A)$, equipped with the $u$-topology. 
Its group of morphisms is the closed subgroup $N(A)\subset \U(\mathfrak{F})$ defined in \cref{LemmaNA}.
The  group homomorphism 
$t_A: N(A) \to \Aut(A)$
from \cref{from-normalizer-to-automorphisms} is the target map. 
We define the source map by
\begin{equation} \label{DefinitionSource}
        s_A: N(A) \to \Aut(A), \qquad s_A(U) := t_A(JUJ) \text{.}
\end{equation}
Clearly, $s_A$ is also a group homomorphism and continuous.
The group homomorphism 
\begin{equation} \label{DefinitionIdentityCanonicalImplementatoin}
  i : \Aut(A) \to N(A)
\end{equation}
assigning the identity morphism is provided by the following classical theorem, see \cite[Thm.~3.2, Prop.~3.5]{HaagerupStandardForm}.

\begin{theorem}
        \label{ThmHaagerupStandardForm}
        For every $\theta \in \Aut(A)$, there exists a unique element $i(\theta) \in N(A)$ such that
        \begin{equation} \label{DefiningPropertiesi}
                i(\theta) a\,  i(\theta)^* = \theta(a), \qquad  i(\theta) J  = J \, i(\theta), \qquad  i(\theta) P = P.
        \end{equation}
        Moreover, the corresponding map $i : \Aut(A) \to N(A)$ is a continuous group homomorphism.\end{theorem}

The group homomorphism $i$ is also called  the \emph{canonical implementation}.
One can see directly from the first equation in \cref{DefiningPropertiesi} that $t_A \circ i$ is the identity on $\Aut(A)$, and the second equation implies that also $s_A \circ i$ is the identity.
\begin{comment}
In particular, $t$ is surjective, i.e., any strongly continuous $*$-automorphism $\theta$ is implemented by some unitary on $F$.
The existence of $i$ shows that we have an isomorphism of topological groups 
\begin{equation*}
\Aut(A) \cong N(A)/\U(A^\prime).
\end{equation*}
Since $\ker(t) = \U(A^\prime)$ is closed in $N(A)$, this implies that $\Aut(A)$, as the quotient of a Polish group by a closed subgroup, is Polish.
We remark that in contrast to $N(A)$ and $t$, the identity morphism $i$ depends on  the choice of $J$ and $P$.
\end{comment}
We further note that $\mathrm{ker}(t_A)\subset A'$ and $\mathrm{ker}(s_A)\subset A''=A$, so that by now all conditions of \cref{LemmaMinimalData2Group} are satisfied, and we obtain the topological 2-group $\mathcal{U}(A)$.

\begin{proposition}
The topological 2-group $\mathcal{U}(A)$ and the crossed module $\UAUT(A)$ correspond to each other under the equivalence \cref{AdjunctionLieCross}; precisely, we have
$\mathscr{X}(\mathcal{U}(A)) = \UAUT(A)$.
\end{proposition}

\begin{proof}
The topological crossed module $\mathscr{X}(\mathcal{U}(A))$ is given by 
\begin{equation*}
{t_A}|_{\ker(s_A)}:N(A) \supset \ker(s_A) \to \Aut(A)\text{,}
\end{equation*}
and $\theta \in \Aut(A)$ acts on $\ker(s_A)$ by conjugation with $i(\theta)$.
We observe that $\ker(s_A)$ consists of those $U \in N(A)$ such that conjugation by $JUJ$ acts trivially on $A$; in other words, $JUJ \in A^\prime$ and $U \in A$.
This shows that $\ker(s_A) = \U(A)$.
Here, by construction of $\AUT(A)$, $\U(A) \subset N(A) \subset \U(F)$ inherits the strong operator topology on $\U(\mathfrak{F})$. This topology coincides with the ultraweak topology of $\U(A) \subset A$.
%
\begin{comment}
  It is well-known that if $F$ is a standard form, any element $\omega$ of the predual $A_*$ has the form $\omega(a) = \langle \xi, a\eta\rangle$, for $\xi, \eta \in F$. 
  The topology on $\U(A)$ is given by the family of seminorms $\nu_\omega(u) = |\omega(u)|$, where $\omega$ ranges over $A_*$, which by the previous remark is equal to the family of seminorms $\nu_{\xi, \eta}(u)  = |\langle \xi, u \eta\rangle|$, where $\xi, \eta$ ranges over $F$. 
  But this family of seminorms induces precisely the weak topology on $\B(F)$, which coincides with the strong topology on $\U(H)$.
\end{comment}
%
By the first property of $i$ in \eqref{DefiningPropertiesi}, the induced action of $\Aut(A)$ on $\ker(s_A) = \U(A)$ is precisely the evaluation action.
Thus, we have $\mathscr{X}(\mathcal{U}(A))=\UAUT(A)$. 
\end{proof}

\begin{remark}
\label{2-group-version-of-UAUT}
Using again the natural isomorphism $\mathcal{G}\circ \mathscr{X}\cong \id$, we obtain a canonical 2-group isomorphism
\begin{equation*}
\mathcal{U}(A) \cong \mathcal{G}(\UAUT(A))\text{.}
\end{equation*}
In particular, on the level of morphisms, the map
\begin{equation*}
 N(A) \to \U(A) \rtimes \Aut(A);\quad U \mapsto (U i(s_A(U))^{-1}, s_A(U)) 
\end{equation*}
is an isomorphism of topological groups. 
\end{remark}

\subsection*{The stringor representation in the  2-group setting}

The following lemma about implementers and the modular conjugation operator will be important, see \cite[\S 4.1]{Kristel2019}. We let $\tau: H \to H$ be induced by the map $t \mapsto 2\pi - t$ on $S^1$, and in turn let $\sigma: \O(H) \to \O(H)$ be defined by $\sigma(g) := \tau\circ \sigma\circ \tau$. The map $\sigma$ restricts to a smooth automorphism of the Banach Lie group $\O_{\mathrm{res}}(H)$. Note that $\sigma(\omega(\gamma_1\cup\gamma_2))=\omega(\gamma_2\cup\gamma_1)$ for any pair $(\gamma_1,\gamma_2) \in P_e\Spin(d)^{[2]}$.   

\begin{lemma}[{{\cite[Prop. 4.9 \& 4.11]{Kristel2019}}}]
\label{implementation-and-J}
If $U \in \U(\mathfrak{F})$ implements $g\in \mathrm{O}_{\mathrm{res}}(H)$, then $JUJ$ implements $\sigma(g)$. The corresponding map $\tilde\sigma: \Imp(H) \to \Imp(H)$ is a Banach Lie group homomorphism, thus making the diagram
\begin{equation*}
\alxydim{}{\Imp(H) \ar[d] \ar[r]^{\tilde\sigma} & \Imp(H) \ar[d] \\ \O_{\mathrm{res}}(H) \ar[r]_{\sigma} & \O_{\mathrm{res}}(H)}
\end{equation*} 
commutative.
\end{lemma}

Now we are in position to  set up the stringor representation as a 2-group homomorphism
\begin{equation*}
\mathcal{R}:\GString(d) \to \mathcal{U}(A)\text{.}
\end{equation*}
It is defined on the level of objects and morphisms by the continuous group homomorphisms\begin{align*}
\mathcal{R}_0 &: \GString(d)_0 \to \Aut^{*}(A): \gamma \mapsto R_0(\gamma)
\\   
\mathcal{R}_1 &: \GString(d)_1 \to N(A): \Phi \mapsto \Omega^\prime(\Phi) %\tilde R_1(\Phi)\text{,}
\end{align*}
where $R_0$ was defined in \cref{definition-of-R0} and $\Omega^\prime$ was defined in \cref{restriction-to-normaliser}.

\begin{theorem}
\label{stringor-representation-for-2-groups}
The group homomorphisms $\mathcal{R}=(\mathcal{R}_0,\mathcal{R}_1)$ form a continuous 2-group homomorphism
\begin{equation*}
\mathcal{R}:\GString(d) \to \mathcal{U}(A)\text{.} 
\end{equation*}
Moreover, we have $\mathscr{X}(\mathcal{R})=\mathscr{R}$, i.e., $\mathcal{R}$ is the 2-group analog of the stringor representation $\mathscr{R}$.
\end{theorem}

\begin{proof}
If $\Phi\in \widetilde{P_e\Spin(d)^{[2]}}$ projects to $\gamma_1 \cup \gamma_2$,
\cref{DiagramNA,definition-of-R0} imply
\begin{equation}
\label{proof-target-respected}
t_{\mathcal{U}(A)}(\mathcal{R}_1(\Phi)) = t_{A}(\Omega^\prime(\Phi)) = \overline{\Cl}^\prime_{\omega(\gamma_1 \cup \gamma_2)} = \overline{\Cl}^\prime_{\omega(\gamma_1 \cup \gamma_1)} = R_0(\gamma_1)=\mathcal{R}_0(t(\Phi)).
\end{equation}
Moreover, \cref{implementation-and-J} shows that
\begin{equation}
\label{proof-source-respected}
s_{\mathcal{U}(A)}(\mathcal{R}_1(\Phi)) = t_{A}(J\Omega^\prime(\Phi)J) = \overline{\Cl}^\prime_{\omega(\gamma_2 \cup \gamma_1)} = R_0(\gamma_2)=\mathcal{R}_0(s(\Phi)).
\end{equation}
This shows that $\mathcal{R}_0$ and $\mathcal{R}_1$ respect sources and targets. 

Next we show that $\mathcal{R}_0$ and $\mathcal{R}_1$ respect the identity morphisms, i.e., that
\begin{equation}
\label{equality-of-identities}
i_{can}(\Rep_0(\gamma)) = \Omega^\prime(i_{ff}(\gamma))
\end{equation}
holds for all $\gamma\in P_e\Spin(d)$,
where $i_{can}$ is the canonical implementation \eqref{DefinitionIdentityCanonicalImplementatoin} and $i_{ff}$ is the fusion factorization \eqref{DefinitionFusionFactorization}.
Since by \cref{ThmHaagerupStandardForm}, elements in the image of $i_{can}$ commute with $J$, we have
\begin{equation*}
 s_{\mathcal{U}(A)}(i_{can}(\Rep_0(\gamma)))= t_{A}(J i_{can}(\Rep_0(\gamma))J) = t_{\mathcal{U}(A)}(i_{can}(\Rep_0(\gamma)))=\Rep_0(\gamma).
 \end{equation*}
Moreover, since $i_{ff}(\gamma)$ projects to $\gamma\cup\gamma$, by \cref{proof-target-respected,proof-source-respected} we get 
\begin{equation*}
s_{\mathcal{U}(A)}( \Omega^\prime(i_{ff}(\gamma)))=t_{\mathcal{U}(A)}( \Omega^\prime(i_{ff}(\gamma)))=R_0(\gamma).
\end{equation*} 
We let $f: P_e\Spin(d) \to N(A)$ be the difference between the two expressions in \cref{equality-of-identities}, i.e.,
\begin{equation*}
f(\gamma) :=i_{can}(\Rep_0(\gamma)) \cdot  \Omega^\prime(i_{ff}(\gamma))^{-1}\text{.} 
\end{equation*}
The above calculations shows that $f(\gamma)\in \mathrm{ker}(s_{\mathcal{U}(A)})\cap \mathrm{ker}(t_{\mathcal{U}(A)})=\U(A')\cap \U(A)=\U(1)$, since $A$ is a factor. 
We observe that we constructed a continuous map
\begin{equation*}
  f: P_e \Spin(d) \longrightarrow \U(1), 
\end{equation*}
which, in fact, is a group homomorphism, as
\begin{equation*}
\begin{aligned}
  f(\gamma \cdot \gamma^\prime) &= i_{can}(\Rep_0(\gamma))\cdot \underbrace{i_{can}(\Rep_0(\gamma^\prime)) \cdot \Omega^\prime(i_{ff}(\gamma^\prime))^{-1}}_{=f(\gamma')\in \U(1)} \cdot\, \Omega^\prime(i_{ff}(\gamma))^{-1} = f(\gamma) \cdot f(\gamma^\prime)\text{.}
\end{aligned}
\end{equation*}
Here we used that the middle term is contained in the center of $N(A)$ and hence can be pulled out.
Since $P_e\Spin(d)$ is a regular Lie group, every continuous group homomorphism is smooth. 
However, by \cite[Thm.~2.1.2]{LudewigWaldorf2Group}, any smooth group homomorphism from $P_e \Spin(d)$ to $\U(1)$ is trivial, $f=1$.
Hence \cref{equality-of-identities} holds.
%
\begin{comment}
Here is another argument using \cref{implementation-and-J}.

By \cref{ThmHaagerupStandardForm}, $i_{can}(\Rep_0(\gamma))$ commutes with $J$. 
We show next that also $\Omega^\prime(i_{ff}(\gamma))$, and hence $f(\gamma)$ commutes with $J$. Indeed, by \cref{implementation-and-J}, conjugation by $J$ induces a  Lie group homomorphism
\begin{equation*}
\tilde{\sigma} : \widetilde{P_e\Spin(d)^{[2]}} \to \widetilde{P_e\Spin(d)^{[2]}}
\end{equation*}
which covers the flip automorphism of $P_e \Spin(d)^{[2]}$ and (by anti-linearity of $J)$ is $\U(1)$-\emph{anti}-equivariant.
By \cite[Lemma~3.3.4]{LudewigWaldorf2Group}, there exists a unique fusion factorization $i$ such that $\tilde{\sigma} \circ i = i$, and by the uniqueness of fusion factorizations (\cite[Thm.~3.3.5]{LudewigWaldorf2Group}) we get $i = i_{ff}$.
This shows that $\Omega^\prime(i_{ff}(\gamma))$ commutes with $J$, for every $\gamma \in P_e \Spin(d)$.

The restriction of the conjugation with $J$ to $\U(1)\subset \mathcal{B}(\mathfrak{F})$ is complex conjugation. Thus, $f(\gamma)\in \{\pm 1\}$. Since $f$ is a continuous map on a connected space, it is constant and determined by its value on the constant path, showing $f(\gamma)=1$. This shows \cref{equality-of-identities}.  
\end{comment}
%
It follows now from \cref{LemmaMinimalData2Group} that $\mathcal{R}$ respects composition and inversion, and hence is a 2-group homomorphism. 

The strict intertwiner $\mathscr{X}(\mathcal{R})$ consists of the group homomorphisms $\mathscr{X}(\mathcal{R})_0=\mathcal{R}_0=R_0$ and $\mathscr{X}(\mathcal{R})_1 = \mathcal{R}_1|_{\mathrm{ker}(s)}= \Omega^\prime|_{\widetilde{\Omega_{[0,\pi]}\Spin(d)}}= R_1$. This shows the claimed equality $\mathscr{X}(\mathcal{R})=\mathscr{R}$.
\end{proof}

\begin{remark}
The equality $\mathscr{X}(\mathcal{R})=\mathscr{R}$ is equivalent to the statement that the diagram of 2-group homomorphisms
\begin{equation*}
\alxydim{}{\GString(d) \ar[r]^-{\mathcal{R}} \ar[d] & \mathcal{U}(A) \ar[d] \\ \mathcal{G}(\XString(d)) \ar[r]_-{\mathcal{G}(\mathscr{R})} & \mathcal{G}(\UAUT(A))}
\end{equation*} 
 whose vertical arrows are the 2-group isomorphisms of \cref{2-group-version-of-string,2-group-version-of-UAUT},  is strictly commutative.  
\end{remark}

\bibliography{bibfile}
\bibliographystyle{kobib}

\end{document}